\documentclass[11pt]{article}
\usepackage[T1]{fontenc}
\usepackage[margin=1in]{geometry}
\usepackage{amsmath,amssymb,amsthm,mathtools}
\usepackage{enumitem}
\usepackage{microtype}
\usepackage[colorlinks=true,linkcolor=blue,citecolor=blue,urlcolor=blue]{hyperref}

\allowdisplaybreaks
\numberwithin{equation}{section}
\theoremstyle{plain}
\newtheorem{theorem}{Theorem}[section]
\newtheorem{lemma}[theorem]{Lemma}
\newtheorem{proposition}[theorem]{Proposition}
\newtheorem{corollary}[theorem]{Corollary}
\newtheorem*{lemma*}{Lemma}
\theoremstyle{definition}
\newtheorem{definition}[theorem]{Definition}
\theoremstyle{remark}
\newtheorem{remark}[theorem]{Remark}

\newcommand{\Rn}{\mathbb{R}^n}
\newcommand{\loc}{\operatorname{loc}}
\newcommand{\supp}{\operatorname{supp}}
\newcommand{\W}{\mathcal{W}}
\newcommand{\dist}{\operatorname{dist}}
\newcommand{\osc}{\operatorname*{osc}}
\newcommand{\minisection}[1]{\par\smallskip\noindent\textbf{#1}\par\smallskip}
\setlist[itemize]{topsep=2pt,itemsep=1pt,parsep=0pt,leftmargin=1.5em}
\setlist[enumerate]{topsep=2pt,itemsep=1pt,parsep=0pt,leftmargin=1.8em}
\AtBeginDocument{%
  \setlength{\abovedisplayskip}{7pt plus 2pt minus 3pt}%
  \setlength{\belowdisplayskip}{7pt plus 2pt minus 3pt}%
  \setlength{\abovedisplayshortskip}{5pt plus 2pt minus 3pt}%
  \setlength{\belowdisplayshortskip}{5pt plus 2pt minus 3pt}%
}
\makeatletter
\renewcommand\section{\@startsection{section}{1}{\z@}%
  {-2.2ex \@plus -.7ex \@minus -.2ex}%
  {1.0ex \@plus .2ex}%
  {\normalfont\Large\bfseries}}
\renewcommand\subsection{\@startsection{subsection}{2}{\z@}%
  {-1.8ex \@plus -.5ex \@minus -.2ex}%
  {.7ex \@plus .2ex}%
  {\normalfont\large\bfseries}}
\renewcommand\subsubsection{\@startsection{subsubsection}{3}{\z@}%
  {-1.4ex \@plus -.4ex \@minus -.2ex}%
  {.5ex \@plus .1ex}%
  {\normalfont\normalsize\bfseries}}
\makeatother

\title{Isolated Singularities of $k$-Hessian Equations}
\author{Shiguang Ma and Shengyang Zang}

\begin{document}
\maketitle

\begin{abstract}
We classify isolated singularities of nonpositive, $k$-admissible
classical solutions to $k$-Hessian equations without radial symmetry
assumptions. For $1\le k<n/2$, we consider $F_k[u]=(-u)^p$ with
$1<p\le nk/(n-2k)$. Below the Serrin exponent, we obtain
fundamental-solution asymptotics and characterize removability by
the vanishing of the leading coefficient. At the Serrin exponent,
we determine the exact logarithmic correction and its universal
coefficient for every nonremovable singularity. In the critical
dimension $k=n/2$, we establish logarithmic asymptotics and a
removability criterion for $F_k[u]=f(x,u,\nabla u)$ with continuous
sources satisfying $0\le f(x,s,\xi)\le C|\xi|^\ell e^{-\alpha s}$,
where $\alpha>0$ and $0\le\ell<n$. The proofs combine Hessian-measure
compactness, Wolff-potential estimates, a global B\^ocher theorem,
and a bounded-variation analysis of spherical extrema.
\end{abstract}

\section{Introduction}\label{sec:introduction}

In this paper, we study the local behavior of nonpositive
solutions to two classes of  $k$-Hessian equations. We assume the dimension $n\ge2$   throughout the paper. For
$1\le k<\frac{n}{2}$, we consider the Lane--Emden type equation
\begin{equation}\label{eq:subcritical-hessian-equation}
	F_k[u]=(-u)^p,
\end{equation}
whereas for $ k=\frac{n}{2}$, we consider
\begin{equation}\label{eq:critical-dimension-equation}
	F_\frac{n}{2}[u]=f(x,u,\nabla u).
\end{equation}
Given a nonpositive solution in a punctured neighborhood of the origin, say $B_1\setminus\{0\}$, the problem is
to determine whether the singularity is removable and, if it is not, to
identify the exact asymptotic behavior of $u$.  For $1\le k<\frac{n}{2}$, we give a
complete classification in the range $1<p\le \frac{nk}{n-2k}$.
For $k=\frac{n}{2}$, we assume the right-hand side $f(x,u,\nabla u)$ satisfies the growth condition \eqref{eq:critical-source-growth}.
     
Let $\Omega\subset\mathbb R^n$ be a bounded domain.  For
$1\le k\le n$ and $u\in C^2(\Omega)$, the $k$-Hessian operator is
defined by
\[
	F_k[u]=\sigma_k\bigl(\lambda(D^2u)\bigr),
	\qquad
	\sigma_k(\lambda)
	=\sum_{1\le i_1<\cdots<i_k\le n}
	\lambda_{i_1}\cdots\lambda_{i_k},
\]
where $\lambda(D^2u)=(\lambda_1,\ldots,\lambda_n)$ denotes the vector of
eigenvalues of $D^2u$.  Thus $F_1[u]=\Delta u$, whereas
$F_n[u]=\det D^2u$.  The associated G\aa rding cone is
\[
\Gamma_k
:=\bigl\{\lambda\in\mathbb R^n:
\sigma_j(\lambda)>0\ \text{for }j=1,\ldots,k\bigr\},
\]
and its closure is denoted by $\overline{\Gamma}_k$. An upper-semicontinuous
function $u:\Omega\to[-\infty,\infty)$ is said to be $k$-subharmonic in $\Omega$, $1\le k\le n$, if 
 $F_k[q]\ge0$  for any quadratic polynomial $q$ such that $u-q$ has a local finite maximum in $\Omega$. We denote by $\Phi^k(\Omega)$ the set of all $k$-subharmonic functions in $\Omega$ which are not equal to $-\infty$ identically.  Trudinger and Wang  \cite{TrudingerWang1997,TrudingerWang1999,TrudingerWang2002}   associate with every
$u\in\Phi^k(\Omega)$ a nonnegative Radon measure, again denoted by
$F_k[u]$.  This measure agrees with
$\sigma_k(\lambda(D^2u))\,dx$ for $C^2$ functions and is weakly
continuous under local $L^1$ convergence.  See
\cite{CNS1985,Wang1994,Trudinger1997,Wang2009KHessian} for the
classical theory and
\cite{TrudingerWang1997,TrudingerWang1999,TrudingerWang2002} for Hessian
measure theory.

For comparison, consider the model case $k=1$. For a nontrivial
solution, setting $v=-u$ in \eqref{eq:subcritical-hessian-equation}
gives the classical Lane--Emden equation
\[
-\Delta v=v^p,
\qquad v>0.
\]
Brezis and Lions \cite{BrezisLions1981,Lions1980} showed that a positive
solution in a punctured domain extends across the puncture in the
distributional sense as
\begin{equation}\label{eq:intro-classical-dirac}
-\Delta v=v^p+a\delta_0,
\qquad a\ge0.
\end{equation}
For $1<p<n/(n-2)$, the coefficient $a$ distinguishes removable
singularities from those with fundamental-solution asymptotics. At and
above the Serrin exponent $n/(n-2)$, this coefficient vanishes, and the
nonlinear equation determines the remaining singular behavior.
The corresponding classifications were established by Aviles
\cite{Aviles1987} at $p=n/(n-2)$, by Gidas and Spruck
\cite{GidasSpruck1981} for $n/(n-2)<p<(n+2)/(n-2)$, and by Caffarelli,
Gidas and Spruck \cite{CaffarelliGidasSpruck1989} for
$n/(n-2)\le p\le(n+2)/(n-2)$.

Serrin's work \cite{Serrin1964,Serrin1965} laid the foundations for the
local theory of isolated singularities of quasilinear equations.
For the $m$-Laplace equation $-\Delta_m v=v^p$, with $1<m<n$,
Guedda and V\'eron \cite{GueddaVeron1988} classified radial singularities
in the range $m-1<p<nm/(n-m)-1$; the range $1<p\le m-1$ falls within
Serrin's theory. Nonradial classifications at and below the Serrin
exponent $n(m-1)/(n-m)$ are developed in our work \cite{MaZang2025}
and its sequel, currently in preparation \cite{MaZang2026}.
These arguments use one-sided differential inequalities for spherical
extrema after the Emden--Fowler transformation.
Analogous questions for the fractional equation $(-\Delta)^s v=v^p$,
$0<s<1$, have been studied below the Serrin exponent
\cite{ChenQuaas2018}, at the Serrin exponent
\cite{WeiWu2022,ChenHajaiej2022,ChanDelaTorre2024}, between the
Serrin and Sobolev exponents \cite{YangZou2021}, and at the Sobolev
exponent \cite{CaffarelliJinSireXiong2014}.
Labutin \cite{Labutin2001} studied isolated singularities for homogeneous
fully nonlinear elliptic equations.

For $k$-Hessian equations, Labutin's potential estimates
\cite{Labutin2002} and the work of Phuc and Verbitsky
\cite{PhucVerbitsky2008} provide a framework for studying equations
with measure data, including solvability, pointwise estimates, and
removable singularities. More recently, Dai, Fu, Gui and Qin
\cite{DaiFuGuiQin2026} established Liouville theorems and classification
results for entire solutions of $k$-Hessian equations of Lane--Emden
type. The present paper concerns precise local asymptotics near an
isolated singularity. At the Serrin exponent, we determine the
logarithmic correction and its universal coefficient, which describe
the nonremovable singularities remaining after the Dirac mass vanishes.
In the critical dimension, we obtain logarithmic asymptotics and a
removability criterion for sources that may also depend on the gradient.

The passage from the Laplace and $m$-Laplace equations to the
$k$-Hessian equation introduces several difficulties. First, although
every smooth $k$-admissible function satisfies
\[
kF_k[u]=\operatorname{div}\bigl(T_{k-1}(D^2u)\nabla u\bigr),
\]
the Newton tensor $T_{k-1}(D^2u)$ depends on $D^2u$. This identity
therefore does not directly yield the first-order energy estimates
used in the Laplace and $m$-Laplace settings. We obtain the required
local bounds through a Wolff-potential bootstrap in
Sections~\ref{sec:subcritical} and~\ref{sec:serrin-exponent}.
Second, the available interior $C^2$ estimates depend on additional
structural hypotheses and do not directly provide the uniform control
needed for our singular rescalings; see, for example,
\cite{ChouWang2001,ShankarYuan2025,WarrenYuan2009,Zhang2025}.
We therefore pass to rescaled limits using local $L^1$ compactness and
the weak continuity of Hessian measures. Finally, an annular Harnack
inequality and a global B\^ocher theorem yield asymptotic radiality.
At the Serrin exponent, we then determine the exact asymptotic
coefficient through one-sided differential inequalities and a
bounded-variation analysis of the spherical minimum and maximum.

We now state our main results.  The first is a B\^ocher type theorem
for punctured $k$-harmonic functions, which identifies the limits arising
in the rescaling arguments.

\begin{theorem}\label{thm:k-hessian-punctured-bocher}
	Let $1\le k\le \frac{n}{2}$. Assume that
	$U\in\Phi^k(\mathbb R^n\setminus\{0\})$ is nonpositive and satisfies
	$F_k[U]=0$. If $1\le k<\frac{n}{2}$, then there exist
	$A\ge0$ and $B\le0$ such that
	\begin{equation}\label{eq:bocher-final-form}
		U(x)=B-A|x|^{2-\frac{n}{k}},
		\qquad x\ne0.
	\end{equation}
	If $k=\frac{n}{2}$, then $U\equiv B$ for some $B\le0$.
\end{theorem}

An isolated singularity is called \emph{removable} if,
for some $\gamma\in(0,1)$, its canonical $k$-convex extension is locally
$C^{0,\gamma}$ across the puncture and satisfies the stated equation there
in the distributional sense.  We write $B_r(x)$ or $B(x,r)$ for the open ball in $\mathbb{R}^n$ with the center $x$ and the radius $r$, and set $B_r=B_r(0)$. The next theorem treats the sub-Serrin range. 
\begin{theorem}
	\label{thm:subcritical-hessian-classification}
	Assume $1\le k<\frac{n}{2}$ and $1<p<\frac{nk}{n-2k}$.  Let
	$u\in C^2(B_2\setminus\{0\})\cap\Phi^k(B_2\setminus\{0\})$ be a
	nonpositive solution of \eqref{eq:subcritical-hessian-equation} in
	$B_1\setminus\{0\}$.  Then there exists $m\ge0$ such that
	\begin{equation}\label{eq:subcritical-fundamental-limit}
		\lim_{x\to0}\frac{u(x)}{-|x|^{2-\frac{n}{k}}}=m.
	\end{equation}
	Moreover, the singularity at the origin is
	removable if and only if $m=0$.
\end{theorem}

At the Serrin exponent, the vanishing of the Dirac mass does not imply removability. The remaining singularity is described by a precise logarithmic correction with a universal coefficient.
\begin{theorem}
	\label{thm:hessian-log-limit}
	Assume $1\le k<\frac{n}{2}$ and $p=\frac{nk}{n-2k}$.  Let
	$u\in C^2(B_2\setminus\{0\})\cap\Phi^k(B_2\setminus\{0\})$ be a
	nonpositive solution of \eqref{eq:subcritical-hessian-equation} in
	$B_1\setminus\{0\}$.  Then exactly one of the following holds:
	\begin{enumerate}[label=\textup{(\roman*)}]
		\item the singularity at the origin is removable;
		\item the singularity is nonremovable; in this case
		$u(x)\to-\infty$ as $x\to0$ and
		\begin{equation}\label{eq:hessian-subfundamental-limit}
			\lim_{x\to0}\frac{u(x)}{-|x|^{2-\frac{n}{k}}}=0.
		\end{equation}
More precisely,
		\begin{equation}\label{eq:hessian-log-limit}
			\lim_{x\to0}
			\left(\log\frac1{|x|}\right)^{\frac{n-2k}{2k^2}}
			|x|^{\frac{n-2k}{k}}(-u(x))
			=\mathcal L_{n,k},
		\end{equation}
where
		\begin{equation}\label{eq:Lnk-def}
			\mathcal L_{n,k}
			=\left[
			\binom{n-1}{k-1}\frac{n-2k}{2k^2}
			\left(\frac{n-2k}{k}\right)^k
			\right]^{\frac{n-2k}{2k^2}}.
		\end{equation}
	\end{enumerate}
\end{theorem}

Finally, consider the critical case $k=\frac{n}{2}$, in which the fundamental
solution is logarithmic.  The right-hand side may depend on both $u$ and
$\nabla u$; in particular, the equation $F_k[u]=e^{-u}$ is included.
\begin{theorem}
	\label{thm:critical-dimension-classification}
	Assume $k=\frac{n}{2}$.  Let
	\[
		f\in C\bigl(B_1\times(-\infty,0]\times\mathbb R^n\bigr),
		\qquad f\ge0,
	\]
	and suppose that, for some $C_0>0$, $\alpha>0$, and
	$0\le\ell<n$,
	\begin{equation}\label{eq:critical-source-growth}
		f(x,s,\xi)\le C_0|\xi|^\ell e^{-\alpha s}
		\qquad
		(x,s,\xi)\in B_1\times(-\infty,0]\times\mathbb R^n,
	\end{equation}
	with the convention $|\xi|^0=1$.  Let
	$u\in C^2(B_2\setminus\{0\})\cap\Phi^k(B_2\setminus\{0\})$ be a
	nonpositive solution of \eqref{eq:critical-dimension-equation} in $B_1\setminus\{0\}$. Then there exists $m\ge0$ such that
	\begin{equation}\label{eq:critical-dimension-limit}
		\lim_{x\to0}\frac{-u(x)}{\log\frac{1}{|x|}}=m.
	\end{equation}
	Moreover, the singularity at the origin is
	removable if and only if $m=0$.
\end{theorem}
The proofs of the three classification theorems begin with a common
measure-theoretic reduction.
By Lemma~\ref{lem:hessian-measure-decomposition-origin}, each solution
admits a canonical $k$-convex extension across the origin, whose
Hessian measure satisfies
\[
F_k[u]=M\delta_0+g(x)\,dx
\quad\text{in }B_1,
\qquad M\ge0,\qquad g\in L^1_{\mathrm{loc}}(B_1).
\]
Here, away from the origin, $g(x)=(-u(x))^p$ for
\eqref{eq:subcritical-hessian-equation}, whereas
$g(x)=f(x,u(x),\nabla u(x))$ in the critical-dimensional case
$k=n/2$. This decomposition establishes local integrability of
the source and isolates any mass concentrated at the puncture.

Theorem~\ref{thm:k-hessian-punctured-bocher} then identifies the
nonpositive $k$-harmonic limits arising in the rescaling arguments.
The subsequent analysis determines the leading coefficients in the
sub-Serrin and critical-dimensional cases, and the precise logarithmic
correction at the Serrin exponent.

This paper is organized as follows.  Section~\ref{sec:prelim-removable}
develops the Hessian measure and Wolff potential tools and proves
Theorem~\ref{thm:k-hessian-punctured-bocher}.
Section~\ref{sec:subcritical} proves the classifications for the sub-Serrin range $1<p<\frac{nk}{n-2k}$,
Section~\ref{sec:serrin-exponent} treats the Serrin exponent $\frac{nk}{n-2k}$, and
Section~\ref{sec:critical-dimension} handles the critical dimension
$k=\frac{n}{2}$.  The technical phase-plane details for the complementary lower estimate at
the Serrin exponent are collected in Appendix~\ref{app:critical-lower-bound}.

\section{Hessian measures and nonlinear potentials}
\label{sec:prelim-removable}

This section develops the measure-theoretic and potential-theoretic
framework used throughout the paper.  We first record the local
integrability and potential estimates needed below, together with the
H\"older regularity criterion above the density threshold.  We then
construct the canonical $k$-convex extension across an isolated
puncture and identify the part of its Hessian measure concentrated
there.  These ingredients are finally combined to prove
Theorem~\ref{thm:k-hessian-punctured-bocher}.

The underlying theory of Hessian measures is due to Trudinger and Wang
\cite{TrudingerWang1997,TrudingerWang1999,TrudingerWang2002}.
Labutin established the corresponding two-sided pointwise estimates in
terms of Wolff potentials \cite{Labutin2002}, while Phuc and Verbitsky
developed the global measure-data theory for Hessian equations of
Lane--Emden type \cite{PhucVerbitsky2008}.

\subsection{Local estimates and the H\"older regularity threshold}

We first fix the notation for Hessian measures and Wolff potentials. Recall that we write $\Phi^k(\Omega)$ for the class of  $k$-convex functions
in $\Omega$, and $F_k[u]$ for the
associated $k$-Hessian measure.  If $u\in C^2(\Omega)$, then this measure
is absolutely continuous and its density is $F_k(D^2u)$:
$
dF_k[u]=F_k(D^2u)\,dx.
$
For $h\in L^1_{\loc}(\Omega)$, the notation
\[
F_k[u]=h\qquad\text{in }\mathcal D'(\Omega)
\]
means
\[
\int_\Omega \varphi\,dF_k[u]
=\int_\Omega \varphi h\,dx,
\qquad \varphi\in C_c^\infty(\Omega),
\]
equivalently, the Hessian measure $F_k[u]$ is $h(x)\,dx$.  For classical
$C^2$ solutions this is the same as $F_k(D^2u)=h$ pointwise, hence also in the sense of distributions.

We shall repeatedly use the following local integrability properties of
$k$-convex functions, established in \cite{TrudingerWang1999}.  If
$u\in \Phi^k(\Omega)$, $u\le0$, $1\le k\le \frac{n}{2}$, and
$\Omega'\Subset\Omega$, then
$u\in L^q$ for all $0\le q<\frac{nk}{n-2k}$ when $1\le k<\frac{n}{2}$, and for all $0\le q<\infty$
 when
$k=\frac{n}{2}$.
Moreover, $\nabla u\in L^q_{\mathrm{loc}}(\Omega)$ for every $0\le q<\frac{nk}{n-k}$.

The pointwise estimates used below are formulated in terms of the
following nonlinear potential. For a nonnegative Borel measure $\mu$ on a domain $\Omega\subset\mathbb R^n$, define the truncated Wolff potential
\[
\W_k^\mu(x,R)
=\int_0^R
\left(\frac{\mu(B(x,t))}{t^{n-2k}}\right)^{\frac{1}{k}}\frac{dt}{t},
\qquad 0<R\le\infty.
\]
Here $\mu(B(x,t))$ means $\mu(B(x,t)\cap\Omega)$.  When
$d\mu=f(y)\,dy$ on $\Omega$, we also write $\W_k^f(x,R)$, and
$\W_k^f(x)=\W_k^f(x,\infty)$.  The potential of the full Hessian measure
is written as $\W_k^{F_k[v]}$.  We use
$\chi_E$ for the characteristic function of $E$.

The first consequence needed below is a regularity criterion when the
density lies above $\frac{n}{2k}$.

\begin{proposition}
\label{thm:general-removable-lm}
Let $1\le k\le \frac{n}{2}$, $u\in\Phi^k(B_1)$, and suppose that
$F_k[u]=g$ in $B_1$, where $g\ge0$ belongs to
$L^m(B_1)$ for some $m>\frac{n}{2k}$.
Then $u\in C^{0,\gamma}_{\loc}(B_1)$ for every $0<\gamma<\min\{1,2-\frac{n}{km}\}$.
\end{proposition}

The proof uses Trudinger's local H\"older estimate.  We record together
with it the companion interior gradient estimate needed later in the
critical rescaling argument.
\begin{lemma*}[Trudinger's interior gradient estimate]
Let $2\le k\le n$, let $U\in C^\infty(B_R(y))$ be $k$-admissible, and
suppose
\[
        F_k[U]=f^k\qquad\text{in }B_R(y),
\]
where $f\ge0$ and $f\in C^{0,1}(B_R(y))$.  Then
\begin{equation}\label{eq:trudinger-gradient-original}
        |\nabla U(y)|
        \le C(n,k)\left(
        R^{-1}\osc_{B_R(y)}U
        +R^2\|\nabla f\|_{L^\infty(B_R(y))}
        \right).
\end{equation}
This is estimate (3.10) of Trudinger \cite[p.~1258]{Trudinger1997}.
\end{lemma*}

\begin{lemma*}[Trudinger's local H\"older estimate]
Let $1\le k\le n$, let $U\in C_{loc}^0(\Omega)$ be a $k$-admissible weak solution of
\[
        F_k[U]=\psi
        \qquad\text{in }\Omega,
\]
where $\psi\ge0$ and $\psi\in L^m_{\loc}(\Omega)$ for some
$m>\frac{n}{2k}$. Then $U\in C_{loc}^\alpha(\Omega)$ for any $\alpha<1$, satisfying $\alpha\le 2-\frac{n}{km}$. Furthermore,  for any ball $B_R(y)\Subset\Omega$ and  concentric sub-ball $B_{\sigma R}(y)$ with $0<\sigma<1$, we have the estimate
\begin{equation}\label{eq:trudinger-holder-original}
        \osc_{B_{\sigma R}(y)} U
        \le C(n,k,m,\alpha)\sigma^\alpha
        \left(
        \osc_{B_R(y)}U
        +R^{2-\frac{n}{km}}\|\psi\|_{L^m(B_R(y))}^{\frac{1}{k}}
        \right).
\end{equation}
This is Theorem~4.1 of Trudinger \cite[p.~1258]{Trudinger1997}.
\end{lemma*}

\begin{proof}[Proof of Proposition~\ref{thm:general-removable-lm}]
	Let $\mu=g\,dx$, and set $ \varepsilon=2k-\frac nm>0$.
	For every ball $B(x,r)\Subset B_1$, H\"older's inequality gives
	\begin{equation}\label{eq:general-measure-growth}
		\mu(B(x,r))
		\le \omega_n^{1-\frac{1}{m}}\|g\|_{L^m(B_1)}
		r^{n(1-\frac{1}{m})}
		=C\,r^{n-2k+\varepsilon}.
	\end{equation}
	Labutin's H\"older criterion \cite[Corollary~3.4]{Labutin2002} therefore
	shows that $u$ is locally H\"older continuous, and in particular finite in $B_1$.  We may now apply Trudinger's local H\"older estimate
	\eqref{eq:trudinger-holder-original} on compactly contained balls.  It gives
	$u\in C^{0,\gamma}_{\loc}(B_1)$ for every
	$0<\gamma<\min\{1,2-\frac{n}{km}\}$, as claimed.
\end{proof}

We next turn to the potential estimates used in the bootstrap and
rescaling arguments.  The first is a strong mapping estimate below the
density threshold.

\begin{lemma}\label{lem:strong-wolff}
	Let $1\le k<\frac{n}{2}$, $1<m<\frac{n}{2k}$, and
	$
	s=\frac{nkm}{n-2km}.
	$
	Then, for every nonnegative $f\in L^m(\Rn)$,
	\begin{equation}\label{eq:strong-wolff}
		\|\W_k^f\|_{L^s(\Rn)}
		\le C(n,k,m)\|f\|_{L^m(\Rn)}^{\frac{1}{k}}.
	\end{equation}
\end{lemma}

\begin{proof}
Set $A=\|f\|_{L^m(\Rn)}$ and let $\mathcal M f$ be the
Hardy--Littlewood maximal
function.  The case $A=0$ is trivial.  For a.e. $x$ with
$0<\mathcal M f(x)<\infty$, split the potential at a radius $\rho>0$.
For $0<t<\rho$,
\[
        \int_{B(x,t)}f\le \omega_n t^n \mathcal M f(x),
\]
and hence
\begin{equation}\label{eq:near-part}
        \int_0^\rho
        \left(\frac{\int_{B(x,t)}f}{t^{n-2k}}\right)^{\frac{1}{k}}\frac{dt}{t}
        \le C(\mathcal M f(x))^{\frac{1}{k}}\rho^2.
\end{equation}
For $t>\rho$, H\"older's inequality gives
\[
        \int_{B(x,t)}f
        \le A\omega_n^{1-\frac{1}{m}}t^{n(1-\frac{1}{m})}.
\]
Since $m<\frac{n}{2k}$, we have
\begin{equation}\label{eq:far-part}
        \int_\rho^\infty
        \left(\frac{\int_{B(x,t)}f}{t^{n-2k}}\right)^{\frac{1}{k}}\frac{dt}{t}
        \le C A^{\frac{1}{k}}\rho^{\frac{2k-\frac{n}{m}}{k}}.
\end{equation}
Choose $\rho=(\frac{A}{\mathcal M f(x)})^{\frac{m}{n}}$.  Combining
\eqref{eq:near-part} and \eqref{eq:far-part} yields
\[
        \W_k^f(x)
        \le C A^{\frac{2m}{n}}(\mathcal M f(x))^\beta,
        \qquad
        \beta=\frac1k-\frac{2m}{n}>0.
\]
Because $s\beta=m$ and $\frac{2m}{n}+\beta=\frac{1}{k}$, the strong $(m,m)$ estimate
for $\mathcal M$ gives
\[
        \|\W_k^f\|_{L^s(\Rn)}
        \le C A^{\frac{2m}{n}}\|\mathcal M f\|_{L^m(\Rn)}^\beta
        \le C A^{\frac{1}{k}}.
\]
 This proves \eqref{eq:strong-wolff}.
\end{proof}

The preceding lemma gives integral control of the Wolff potential.  We
shall combine it with Labutin's pointwise comparison between a
$k$-convex function and the Wolff potential of its Hessian measure.

\begin{lemma}\label{lem:labutin-local-pointwise}
	Let $u\in\Phi^k(B_{3R})$, $1\le k\le \frac n2$, $u\le 0$, $F_k[u]=\mu$. Then
	\begin{equation}\label{eq:pointwise-labutin}
		C_1\W_k^\mu(0,\frac R8)\le |u(0)|\le C_2\W_k^\mu(0,2R)+C_3\inf_{B_R}(-u)
	\end{equation}
	If in addition,
	$u\in L_{loc}^q(B_{3R})$ for some $q>1$, then
	\begin{equation}\label{eq:pointwise-start}
		|u(0)|
		\le C_2\W_k^\mu(0,2R)
		+C_3(\omega_nR^n)^{-\frac{1}{q}}\|u\|_{L^q(B_R)}.
	\end{equation}
\end{lemma}

\begin{proof}
The two inequalities in \eqref{eq:pointwise-labutin} are Labutin's  pointwise Wolff potential estimate \cite[Theorem~2.1]{Labutin2002}. 
If $u\in L^q(B_R)$, $q>1$, then we have
\[
\begin{aligned}
        \inf_{B_R}(-u)
        &\le \frac1{|B_R|}\int_{B_R}|u|\,dy  \\
        &\le |B_R|^{-\frac{1}{q}}\|u\|_{L^q(B_R)}
        = (\omega_nR^n)^{-\frac{1}{q}}\|u\|_{L^q(B_R)}.
\end{aligned}
\]
Substituting this into \eqref{eq:pointwise-labutin} gives
\eqref{eq:pointwise-start}.
\end{proof}

\subsection{Canonical extension and measure decomposition}
\label{app:measure-decomposition}
The preceding estimates concern $k$-convex functions defined on a full
ball.  For a solution given only on a punctured ball, we must first
construct a canonical extension and identify the part of its Hessian
measure supported at the puncture.  The next lemma provides both.
\begin{lemma}\label{lem:hessian-measure-decomposition-origin}
Let $1\le k\le \frac{n}{2}$, and suppose
$u\in\Phi^k(B_R\setminus\{0\})$ is nonpositive.  Assume that
$F_k[u]=g\,dx$ as measures on $B_R\setminus\{0\}$, for some
nonnegative $g\in L^1_{\loc}(B_R\setminus\{0\})$.
Then the canonical extension
\[
        u(0):=\limsup_{x\to0}u(x)
\]
belongs to $\Phi^k(B_R)$.  Moreover, after setting $g(0)=0$, we have
$g\in L^1_{\loc}(B_R)$, and the extended Hessian measure satisfies
\begin{equation}\label{eq:hessian-measure-decomposition-origin}
        F_k[u]=M\delta_0+g\,dx
        \quad\text{ in }B_R,
\end{equation}
where $M=F_k[u](\{0\})\ge0$.
\end{lemma}

(The density assumption holds in particular when
$u\in C^2(B_R\setminus\{0\})$, with $g=F_k(D^2u)$, and when
$F_k[u]=0$ on the punctured ball, with $g=0$.)

\begin{proof}
Set
\[
        G_{n,k}(x)=
        \begin{cases}
        -|x|^{2-\frac{n}{k}},&1\le k<\frac{n}{2},\\[1mm]
        \log|x|,&k=\frac{n}{2},
        \end{cases}
        \qquad x\ne0.
\]
The function $G_{n,k}\in\Phi^k(\mathbb R^n)$ is locally
integrable  and tends to $-\infty$ at the origin.  For $\varepsilon>0$,
define
\[
        u_\varepsilon(x)=u(x)+\varepsilon G_{n,k}(x),\quad x\ne0,
        \qquad u_\varepsilon(0)=-\infty.
\]
The class $\Phi^k$ is closed under addition and multiplication by
positive constants \cite[Lemma~2.5]{TrudingerWang1999}, so
$u_\varepsilon\in\Phi^k(B_R\setminus\{0\})$.
Since $u\le 0$, $u_\varepsilon\le\varepsilon G_{n,k}$, and therefore
$u_\varepsilon(x)\to-\infty$ as $x\to0$.  Thus, the extension is
upper semicontinuous at the origin, where the viscosity test is vacuous;
hence $u_\varepsilon\in\Phi^k(B_R)$.

Fix $\varepsilon_0>0$.  Since $u_{\varepsilon_0}$ and $G_{n,k}$ are
locally integrable and
\[
        u=u_{\varepsilon_0}-\varepsilon_0G_{n,k}
        \qquad\text{a.e. in }B_R,
\]
we have $u\in L^1_{\loc}(B_R)$.  Furthermore, for every compact
$K\Subset B_R$,
\[
        \|u_\varepsilon-u\|_{L^1(K)}
        =\varepsilon\|G_{n,k}\|_{L^1(K)}\longrightarrow0.
\]
Passing to the limit in the distributional characterization of
$k$-convexity \cite[Lemma~2.2]{TrudingerWang1999}, we obtain a proper
$k$-convex representative $\widetilde u$, equal to $u$ almost
everywhere.  On $B_R\setminus\{0\}$, both functions are subharmonic;
their values are recovered as limits of averages over shrinking balls.
Consequently, they agree pointwise there.  Upper semicontinuity gives
\[
        \widetilde u(0)\ge\limsup_{x\to0}u(x).
\]
Conversely, since every $k$-convex function is subharmonic, for $0<r<R$,
\[
        \widetilde u(0)
        \le \frac1{|B_r|}\int_{B_r}\widetilde u\,dx
        =\frac1{|B_r|}\int_{B_r}u\,dx
        \le \sup_{0<|x|<r}u(x).
\]
Letting $r\to0$ proves the reverse inequality.  Thus
$\widetilde u$ is the canonical extension, and we again denote it by $u$.

Let $\mu=F_k[u]$ and set $M=\mu(\{0\})$.  Since $\mu$ is a locally
finite Radon measure, $0\le M<\infty$.  The measure
\[
        \nu:=\mu-M\delta_0
\]
is nonnegative and has no Dirac mass at the origin.  By locality of
Hessian measures and the assumed equation on the punctured ball,
\[
        \nu=g\,dx
        \qquad\text{in }B_R\setminus\{0\}.
\]
Set $g(0)=0$.  Since both $\nu$ and $g\,dx$ assign zero mass to
$\{0\}$, the preceding identity extends to every Borel subset of
$B_R$; hence $\nu=g\,dx$ on $B_R$.  Local finiteness of $\mu$ gives
$g\in L^1_{\loc}(B_R)$, and
\eqref{eq:hessian-measure-decomposition-origin} follows.
\end{proof}

\subsection{The global B\^ocher theorem}
\label{sec:bocher}

The preceding decomposition reduces the punctured $k$-harmonic problem
to a global equation whose Hessian measure is supported at the origin.
We now identify the corresponding solutions.  The cases
$1\le k<\frac{n}{2}$ and $k=\frac{n}{2}$ are treated separately because
the Wolff potential of a point mass has power and logarithmic behavior,
respectively.

\begin{proof}[Proof of Theorem~\ref{thm:k-hessian-punctured-bocher}]
Apply Lemma~\ref{lem:hessian-measure-decomposition-origin} on each ball
$B_R$, using $F_k[U]=0$ on $B_R\setminus\{0\}$, so that
$g=0$.  The canonical extensions agree on overlapping balls; hence
$U\in\Phi^k(\mathbb R^n)$ and
\begin{equation}\label{eq:bocher-measure-pointmass-corrected}
        F_k[U]=M\delta_0
        \qquad\text{in }\mathbb R^n
\end{equation}
for some $M\ge0$.
Assume first that $k=\frac{n}{2}$.  If $M>0$, fix $x_0\ne0$ with
$U(x_0)>-\infty$. By Lemma~\ref{lem:labutin-local-pointwise}, we have
\[
\begin{aligned}
        -U(x_0)
        &\ge C(n,k)^{-1}\W_k^{M\delta_0}(x_0,\frac{R}{32})\\
        &\ge C(n,k)^{-1}M^{\frac{1}{k}}
              \int_{2|x_0|}^{\frac{R}{32}}\frac{dt}{t}
        =C(n,k)^{-1}M^{\frac{1}{k}}\log\frac{R}{64|x_0|}.
\end{aligned}
\]
Letting $R\to\infty$ contradicts the finiteness of $U(x_0)$.
Hence $M=0$.

Set $B=\sup_{\mathbb R^n\setminus\{0\}}U$ and $V=B-U\ge0$.
Then $F_k[-V]=0$ and $\inf V=0$.  Fix $x\ne0$, choose $y_j$ with
$V(y_j)\to0$, and take $\rho_j>8|x-y_j|$.  Applying
Lemma~\ref{lem:labutin-local-pointwise} to  $-V$ gives
\[
        V(x)\le C\inf_{B(x,\frac{\rho_j}{8})}V\le C V(y_j).
\]
Thus $V(x)=0$, and $U\equiv B$.

Now let $1\le k<\frac{n}{2}$.  With the same definitions of $B$ and $V$, the global Wolff estimate \cite[Corollary~7.3]{PhucVerbitsky2008}
gives
\[
        0\le V(x)
        \le C\W_k^{M\delta_0}(x,\infty)
        =C M^{\frac{1}{k}}|x|^{2-\frac{n}{k}}.
\]
Thus $U(x)\to B$ as $|x|\to\infty$; if $M=0$, then $U\equiv B$.
For $M>0$, by a uniqueness result
\cite[Theorem~4.5 and (4.15)]{TrudingerWang2002}, we conclude that
\[
        U(x)-B=-a_{n,k}M^{\frac{1}{k}}|x|^{2-\frac{n}{k}},
\]
where
\begin{equation}\label{eq:a-n-def}
        a_{n,k}=\frac1{(\frac{n}{k}-2)(\binom nk\omega_n)^{\frac{1}{k}}}.
\end{equation}
This is the asserted form with $A=a_{n,k}M^{\frac{1}{k}}$.
\end{proof}

\section{The sub-Serrin range}
\label{sec:subcritical}
The proof proceeds from weak critical integrability to a finite
Wolff-potential bootstrap and then to an upper bound by the fundamental
solution.  The classification is completed by applying these estimates to
the mass at the origin in the canonical Hessian measure decomposition.

\subsection{Weak critical integrability}
We begin with a weak estimate at the critical exponent
$\frac{nk}{n-2k}$.  The first lemma concerns the Wolff potential of a
finite measure, and the second transfers this estimate to $k$-convex
functions.
\begin{lemma}\label{lem:finite-measure-weak-wolff}
Let $\mu$ be a finite nonnegative Borel measure on $\mathbb R^n$, with
$1\le k<\frac{n}{2}$.  Then
\begin{equation}\label{eq:finite-measure-weak-wolff}
        \|\W_k^\mu(\cdot,\infty)\|_{L^{\frac{nk}{n-2k},\infty}(\mathbb R^n)}
        \le C(n,k)\,\mu(\mathbb R^n)^{\frac{1}{k}}.
\end{equation}
\end{lemma}

\begin{lemma}
\label{lem:hessian-weak-lorentz-subcritical}
Let $1\le k<\frac{n}{2}$, $w\le0$ and  $w\in \Phi^k(B_2)$. Then
\begin{equation}\label{eq:hessian-weak-critical-quantitative}
        \|w\|_{L^{\frac{nk}{n-2k},\infty}(B_{\frac{3}{2}})}
        \le C(n,k)\left(
        F_k[w](B_{\frac{7}{4}})^{\frac{1}{k}}
        +\|w\|_{L^1(B_{\frac{7}{4}})}
        \right).
\end{equation}
\end{lemma}

\begin{proof}[Proof of Lemma~\ref{lem:finite-measure-weak-wolff}]
Let $m_\mu=\mu(\mathbb R^n)$ and $\mathcal M\mu$ be the centered
maximal function of the measure.  The standard weak $(1,1)$ estimate is
\begin{equation}\label{eq:maximal-measure-weak11}
        |\{\mathcal M\mu>s\}|\le C(n)m_\mu s^{-1}.
\end{equation}
For $m_\mu>0$, splitting the Wolff integral at $\rho>0$ gives
\[
\W_k^\mu(x,\infty)
\le C\left((\mathcal M\mu(x))^{\frac{1}{k}}\rho^2
       +m_\mu^{\frac{1}{k}}\rho^{-\frac{n-2k}{k}}\right).
\]
Choosing $\rho=(\frac{m_\mu}{\mathcal M\mu(x)})^{\frac{1}{n}}$ yields
\begin{equation}\label{eq:hedberg-finite-wolff}
        \W_k^\mu(x,\infty)
        \le C(n,k)m_\mu^{\frac{2}{n}}
        (\mathcal M\mu(x))^{\frac{n-2k}{nk}}.
\end{equation}
Combining this with \eqref{eq:maximal-measure-weak11} and
\[
        1+\frac{2k}{n-2k}=\frac{n}{n-2k}
\]
gives \eqref{eq:finite-measure-weak-wolff}.  The case $m_\mu=0$ is
immediate.
\end{proof}

\begin{proof}[Proof of Lemma~\ref{lem:hessian-weak-lorentz-subcritical}]
Let $\mu$ be the Hessian measure $F_k[w]$ on $B_{\frac{7}{4}}$,
regarded as a measure on $\mathbb R^n$ by zero extension.
For every $x\in B_{\frac{3}{2}}$, by Lemma~\ref{lem:labutin-local-pointwise}
\begin{equation}\label{eq:weak-critical-labutin-input}
    -w(x)
    \le C(n,k)\W_k^\mu(x,\frac{1}{8})
    +C(n,k)\inf_{B(x,\frac{1}{16})}(-w).
\end{equation}
Since $B(x,\frac{1}{16})\subset B_{\frac{7}{4}}$ and $w\in L^1_{\loc}(B_2)$,
\[
\inf_{B(x,\frac{1}{16})}(-w)
\le \frac1{|B_{\frac{1}{16}}|}\int_{B_{\frac{7}{4}}}(-w)
\le C(n)\|w\|_{L^1(B_{\frac{7}{4}})}.
\]
Moreover,
\[
\W_k^\mu(x,\frac{1}{8})\le \W_k^\mu(x,\infty).
\]
 On the other hand, Lemma~\ref{lem:finite-measure-weak-wolff} yields
\[
\|\W_k^\mu(\cdot,\frac{1}{8})\|_{L^{\frac{nk}{n-2k},\infty}(\mathbb R^n)}
\le C(n,k)F_k[w](B_{\frac{7}{4}})^{\frac{1}{k}},
\]
Therefore, for $A\ge 0$, 
\[
\|f+A\|_{L^{r,\infty}(E)}
\le 2\|f\|_{L^{r,\infty}(E)}+2A|E|^{\frac{1}{r}},
\]
and we obtain \eqref{eq:hessian-weak-critical-quantitative} from
\eqref{eq:weak-critical-labutin-input}. 
\end{proof}

\subsection{A finite Wolff bootstrap}

The preceding weak estimate gives $L^{q_0}$-integrability on smaller
balls for every $q_0<\frac{nk}{n-2k}$.  In the sub-Serrin range, a
finite iteration of the local Wolff estimate produces an exponent
$q>Q_p=\frac{pn}{2k}$.

\begin{lemma}
\label{lem:subcritical-wolff-crossing}
Let
\[
        1<p<\frac{nk}{n-2k},\qquad
        Q_p=\frac{pn}{2k},\qquad
        p<q_0<\frac{nk}{n-2k}.
\]
Let $B_{R'}\Subset B_R$ and  $0\le a\le1$. Suppose that
$U<0$ belongs to $\Phi^k(B_R)\cap L^{q_0}(B_R)$ and satisfies
\begin{equation}\label{eq:subcritical-iteration-equation}
        F_k[U]=a(-U)^p\qquad\text{in }B_R.
\end{equation}
Then $U\in L^q(B_{R'})$ for some $q>Q_p$ depending only on
$n,k,p,q_0$.
\end{lemma}

\begin{proof}
The case $q_0>Q_p$ is immediate.  If $q_0\le Q_p$, the
embedding $L^{q_0}(B_R)\hookrightarrow L^{\widetilde q}(B_R)$,
$\widetilde q<q_0$, allows us, after decreasing $q_0$ only in the equality
case, to assume without loss of generality
that $p<q_0<Q_p$. For $p<q<Q_p$, set
\begin{equation}\label{eq:subcritical-T-map}
        T(q)=\frac{nkq}{pn-2kq}.
\end{equation}
Let $B_s\Subset B_r\Subset B_R$ be concentric and suppose
$U\in L^q(B_r)$.  Since $1<\frac{q}{p}<\frac{n}{2k}$,
Lemma~\ref{lem:labutin-local-pointwise}, H\"older's inequality, and Lemma~\ref{lem:strong-wolff} applied to
the source $a(-U)^p$ on $B_r$, with $m=\frac{q}{p}$, give
\begin{equation}\label{eq:subcritical-one-step}
        \|U\|_{L^{T(q)}(B_s)}
        \le C(n,k,p,q,r,s)
        \left(\|U\|_{L^q(B_r)}^{\frac{p}{k}}+\|U\|_{L^q(B_r)}\right).
\end{equation}
Since $T(q)-q$ is increasing,  if
$q_{j+1}=T(q_j)$ while $q_j<Q_p$, then
\[
        q_{j+1}-q_j\ge T(q_0)-q_0>0,
\]
and some finite iterate satisfies $q_J\ge Q_p$.  If $q_J=Q_p$, we again decrease this exponent to a fixed
$p<\widehat q<Q_p$ with $T(\widehat q)>Q_p$, and one further application
of \eqref{eq:subcritical-one-step} crosses $Q_p$.

Iterating \eqref{eq:subcritical-one-step}  finite many times yields $U\in L^q(B_{R'})$ for some $q>Q_p$. 
\end{proof}

\subsection{A fundamental-solution upper bound}

\begin{proposition}
\label{prop:subcritical-fundamental-bound}
Under the assumptions of Theorem~\ref{thm:subcritical-hessian-classification},
there exist $C>0$ and $r_0>0$ such that
\begin{equation}\label{eq:subcritical-uniform-fundamental-bound}
        -u(x)\le C|x|^{2-\frac{n}{k}},
        \qquad 0<|x|<r_0.
\end{equation}
\end{proposition}

\begin{proof}
Suppose for contradiction that \eqref{eq:subcritical-uniform-fundamental-bound}
fails.  Then there is  $x_j\to0$ such that
\[
 |x_j|^{\frac{n-2k}{k}} u(x_j)\longrightarrow-\infty.
 \]
Let $ r_j=\frac{|x_j|}{4}$ and $U_j(\xi)=r_j^{\frac{n-2k}{k}} u(x_j+r_j\xi)$.
Therefore, we have
\[
F_k[U_j]=r_j^{n-\frac{n-2k}{k} p}(-U_j)^p
\qquad\text{in }B_2,
\]
and let $\mu_j$ be the measure on $B_2$ defined by
\[
        d\mu_j=r_j^{n-\frac{n-2k}{k}p}(-U_j)^p\,d\xi.
\]
Then
\begin{equation}\label{eq:subcritical-Uj-zero-blowup}
        U_j(0)=r_j^{\frac{n-2k}{k}} u(x_j)
        =4^{-\frac{n-2k}{k}}|x_j|^{\frac{n-2k}{k}} u(x_j)\to-\infty.
\end{equation}
Set $s=\frac{nk}{n-2k}$.  By
Lemma~\ref{lem:hessian-weak-lorentz-subcritical},
$u\in L^{s,\infty}(B_1)$; hence, for fixed $0<q<s$ and large $j$,
\begin{equation}\label{eq:subcritical-scaled-starting-Lq}
\begin{aligned}
        \|U_j\|_{L^q(B_2)}^q
        &=r_j^{\frac{n-2k}{k} q-n}
          \int_{B(x_j,2r_j)}(-u(y))^q\,dy\\
        &\le C r_j^{\frac{n-2k}{k}q-n}
          \|u\|_{L^{s,\infty}(B_1)}^q
          |B(x_j,2r_j)|^{1-\frac qs}
        \le C(n,k,q,u).
\end{aligned}
\end{equation}
Since $1<\frac{nk}{n-2k}$, taking $q=1$ in
\eqref{eq:subcritical-scaled-starting-Lq} gives
\begin{equation}\label{eq:subcritical-inf-bound-scaled}
        \inf_{B_{\frac{1}{4}}}(-U_j)
        \le \frac1{|B_{\frac{1}{4}}|}\int_{B_{\frac{1}{4}}}(-U_j)
        \le C,
\end{equation}
where $C$ is independent of $j$.
By Lemma~\ref{lem:labutin-local-pointwise},
\[
-U_j(0)
\le C\W_k^{\mu_j}(0,\frac{1}{2})+C\inf_{B_{\frac{1}{4}}}(-U_j).
\]
Combining \eqref{eq:subcritical-Uj-zero-blowup} and
\eqref{eq:subcritical-inf-bound-scaled}, we obtain
\begin{equation}\label{eq:subcritical-Wolff-must-blow}
        \W_k^{\mu_j}(0,\frac{1}{2})\longrightarrow\infty.
\end{equation}
The following uniform  bound gives the required contradiction.

Fix $p<q_0<s$ and nested balls from $B_2$ to $B_1$.
Since $a_j=r_j^{n-\frac{n-2k}{k}p}\le1$ for large $j$, the finite
iteration of \eqref{eq:subcritical-one-step} in
Lemma~\ref{lem:subcritical-wolff-crossing}, starting from
\eqref{eq:subcritical-scaled-starting-Lq}, yields some
$q>Q_p=\frac{pn}{2k}$ and $C$ independent of $j$ such that
\begin{equation}\label{eq:subcritical-scaled-high-integrability}
        \|U_j\|_{L^q(B_1)}\le C.
\end{equation}
Set $ \varepsilon=2k-\frac{np}{q}>0$.
For $0<t<\frac{1}{2}$, H\"older's inequality and
\eqref{eq:subcritical-scaled-high-integrability} imply
\begin{align*}
	\mu_j(B_t)
	&=r_j^{n-\frac{n-2k}{k}p}\int_{B_t}(-U_j)^p\,d\xi                    \\
	&\le C |B_t|^{1-\frac{p}{q}}
	\le C t^{n(1-\frac{p}{q})}
	=C t^{n-2k+\varepsilon}.
\end{align*}
Therefore
\begin{align*}
	\W_k^{\mu_j}(0,\frac{1}{2})
	&=\int_0^{\frac{1}{2}}
	\left(\frac{\mu_j(B_t)}{t^{n-2k}}\right)^{\frac{1}{k}}\frac{dt}{t} \\
	&\le C\int_0^{\frac{1}{2}}t^{\frac{\varepsilon}{k}}\frac{dt}{t}
	<\infty,
\end{align*}
which contradicts
\eqref{eq:subcritical-Wolff-must-blow} and proves the proposition.
\end{proof}

\subsection{The mass at the puncture and the classification}

We now apply the preceding estimates to the canonical Hessian measure
decomposition and distinguish the cases $M=0$ and $M>0$.

\begin{proof}[Proof of Theorem~\ref{thm:subcritical-hessian-classification}]
By Lemma~\ref{lem:hessian-measure-decomposition-origin}, we have
\begin{equation}\label{eq:subcritical-measure-decomposition}
        F_k[u]=M\delta_0+(-u)^p
        \qquad\text{in }\mathcal{D}'(B_1),
\end{equation}
where $M\ge0$ and $(-u)^p\in L^1(B_1)$. Consider first the case $M=0$.
Choose $q_0$ as in Lemma~\ref{lem:subcritical-wolff-crossing}.
Lemma~\ref{lem:hessian-weak-lorentz-subcritical} then gives
$u\in L^{q_0}$ on a small ball.  Applying
Lemma~\ref{lem:subcritical-wolff-crossing} with $a=1$ yields an exponent
$q>Q_p=\frac{pn}{2k}$ and a smaller ball $B_r$ such that $u\in L^q(B_r)$.
Proposition~\ref{thm:general-removable-lm} therefore
shows that the extension is locally $C^{0,\gamma}$ at the origin for
some $\gamma\in(0,1)$.
Consequently, $u$ is bounded near the origin, and hence
\[
        \lim_{x\to0}\frac{-u(x)}{|x|^{2-\frac{n}{k}}}=0.
\]
It remains to assume $M>0$.
By Proposition~\ref{prop:subcritical-fundamental-bound}, after decreasing the
radius if necessary,
\begin{equation}\label{eq:subcritical-pointwise-upper-for-theorem}
	0\le -u(x)\le C|x|^{-\frac{n-2k}{k}},
	\qquad 0<|x|<\frac{1}{2}.
\end{equation}
Since $\frac{n-2k}{k}p<n$, we have
\begin{equation}\label{eq:subcritical-source-mass-vanishes}
	\int_{B_r}(-u)^p\,dx
	\le C\int_0^r s^{n-1-\frac{n-2k}{k}p}\,ds
	=C r^{n-\frac{n-2k}{k}p}\longrightarrow0.
\end{equation}

Let $r_j\to0$ be arbitrary, and define
\[
U_j(\xi):=r_j^{\frac{n-2k}{k}}u(r_j\xi),
\qquad \xi\in B_{1/r_j}.
\]
Then $U_j\in\Phi^k(B_{1/r_j})$, $U_j\le0$, and
\begin{equation}\label{eq:subcritical-scaled-measure-equation}
	F_k[U_j]
	=M\delta_0+a_j(-U_j)^p
	\qquad\text{in }B_{1/r_j},
	\qquad
	a_j:=r_j^{n-\frac{n-2k}{k}p}\longrightarrow0.
\end{equation}

Fix $K\Subset\mathbb R^n\setminus\{0\}$.  For all sufficiently large
$j$, we have $K\Subset B_{1/r_j}$.  By
\eqref{eq:subcritical-pointwise-upper-for-theorem}, the functions $U_j$
are uniformly bounded on a fixed neighborhood of $K$.  On this
neighborhood,
\[
F_k[U_j]=a_j(-U_j)^p,
\]
and hence $F_k[U_j]\to0$ uniformly, since $a_j\to0$.  The local
H\"older estimate \eqref{eq:trudinger-holder-original} therefore yields
equicontinuity of $\{U_j\}$ on $K$.  Since $K$ is arbitrary, a
diagonal argument gives, after passing to a subsequence,
\[
U_j\longrightarrow U
\qquad\text{locally uniformly in }
\mathbb R^n\setminus\{0\}.
\]
Moreover,
\begin{equation}\label{eq:subcritical-limit-fundamental-bound}
	0\le -U(\xi)
	\le C|\xi|^{-\frac{n-2k}{k}}.
\end{equation}

The preceding bound improves the convergence to
$L^1_{\loc}(\mathbb R^n)$.  Indeed, the convergence is uniform on
$B_R\setminus B_\rho$, while on $B_\rho$ both $|U_j|$ and $|U|$
are bounded by $C|\xi|^{-\frac{n-2k}{k}}$, and
\[
\int_{B_\rho}|\xi|^{-\frac{n-2k}{k}}\,d\xi
=C(n,k)\rho^{n-\frac{n-2k}{k}}
\longrightarrow0.
\]
Thus $U_j\to U$ in $L^1_{loc}(\mathbb{R}^n)$. Let $\varphi\in C_c^\infty(\mathbb R^n)$, and choose $R>0$ such that
$\supp\varphi\subset B_R$.  By a change of variables,
\begin{align*}
	\left|
	\int_{B_{1/r_j}}
	\varphi(\xi)a_j(-U_j(\xi))^p\,d\xi
	\right|
	&=
	\left|
	\int_{B_1}
	\varphi\!\left(\frac{x}{r_j}\right)(-u(x))^p\,dx
	\right| \\
	&\le
	\|\varphi\|_{L^\infty}
	\int_{B_{Rr_j}}(-u)^p\,dx
	\longrightarrow0.
\end{align*}
Since $U_j\to U$ in $L^1_{\loc}(\mathbb R^n)$, the local closure of
$\Phi^k$ and the weak continuity of Hessian measures give
\[
U\in\Phi^k(\mathbb R^n),
\qquad
F_k[U_j]\rightharpoonup F_k[U]
\quad\text{locally in }\mathbb R^n.
\]
Combining this with
\eqref{eq:subcritical-scaled-measure-equation}, we obtain
\[
F_k[U]=M\delta_0
\qquad\text{in }\mathbb R^n.
\]

By Theorem~\ref{thm:k-hessian-punctured-bocher} and the identification of
the coefficient in its proof,
\[
U(\xi)
=B-a_{n,k}M^{\frac1k}
|\xi|^{-\frac{n-2k}{k}}.
\]
It follows from \eqref{eq:subcritical-limit-fundamental-bound} that
$B=0$.  Setting
\[
m:=a_{n,k}M^{\frac1k}>0,
\]
we conclude that every sequence $r_j\to0$ has a subsequence such that
\[
U_j(\xi)
\longrightarrow
-m|\xi|^{-\frac{n-2k}{k}}
\qquad\text{locally uniformly in }
\mathbb R^n\setminus\{0\}.
\]
Since $m$ is independent of the chosen sequence, the subsequential
limit is unique.  In particular,
\[
r^{\frac{n-2k}{k}}u(r\theta)
\longrightarrow -m
\qquad\text{uniformly for }
\theta\in\mathbb S^{n-1}.
\]
Equivalently,
\[
\lim_{x\to0}
\frac{u(x)}{-|x|^{2-\frac{n}{k}}}
=m,
\]
which is \eqref{eq:subcritical-fundamental-limit}.  Since $m>0$,
combining this conclusion with the case $M=0$ proves that $m=0$ if
and only if the singularity is removable.
\end{proof}

\section{The Serrin critical exponent}
\label{sec:serrin-exponent}

Throughout this section, unless otherwise stated, $u$ denotes the
solution in Theorem~\ref{thm:hessian-log-limit}.  The proof proceeds
through universal a priori estimates, the exclusion of a Dirac
mass at the origin, asymptotic radiality on shrinking annuli, and the
determination of the exact logarithmic behavior from the spherical
extrema. Recall the asymptotic coefficient $\mathcal L_{n,k}$ from
\eqref{eq:Lnk-def}.  For $0<r<1$, we use the notation
\[
\underline{u}(r)=\inf_{|x|=r}(-u(x)),
\qquad
\overline{u}(r)=\sup_{|x|=r}(-u(x)).
\]

\subsection{A universal  estimate}
\label{sec:critical-fundamental-bound}

We begin with a pointwise estimate for the solution and its gradient in
terms of the distance to the boundary.  Its proof combines the doubling
lemma, compactness in the Hessian-measure topology, the Liouville theorem
of Phuc--Verbitsky, and the two interior estimates of Trudinger recorded
in Section~\ref{sec:prelim-removable}.

\begin{theorem}\label{thm:serrin-fundamental-bound}
Let $1\le k<\frac{n}{2}$.  There exists $C=C(n,k)>0$ such that, for every proper
domain $\Omega\subsetneq\Rn$, each nonpositive solution
$u\in C^2(\Omega)\cap\Phi^k(\Omega)$ of \eqref{eq:subcritical-hessian-equation} with $p=\frac{nk}{n-2k}$,
satisfies
\begin{equation}\label{eq:huang-xu-type-gradient-bound}
        |u(x)|^{\frac{k}{n-2k}}
        +|\nabla u(x)|^{\frac{k}{n-k}}
        \le C\,\dist(x,\partial\Omega)^{-1},
        \qquad x\in\Omega.
\end{equation}
Consequently, if $\Omega=B_1\setminus\{0\}$, then
\begin{equation}\label{eq:serrin-fundamental-bound}
        |u(x)|+|\nabla u(x)|^{\frac{n-2k}{n-k}}
        \le C|x|^{-\frac{n-2k}{k}}
        \qquad\text{for }0<|x|<\frac12.
\end{equation}
\end{theorem}

\begin{proof}
\medskip
\noindent\textbf{Step 1: doubling and rescaling.}
Set $\alpha:=\frac{n-2k}{k}$. Assume that \eqref{eq:huang-xu-type-gradient-bound} fails.  Then there
exist proper domains $\Omega_j\subsetneq\Rn$, nonpositive solutions $u_j\in C^2(\Omega_j)\cap\Phi^k(\Omega_j)$
of \eqref{eq:subcritical-hessian-equation} with
$p=\frac{nk}{n-2k}$, and points $y_j\in\Omega_j$, such that the
functions
\[
M_j(x):=(-u_j(x))^{\frac1\alpha}
+|\nabla u_j(x)|^{\frac1{\alpha+1}},
\qquad x\in\Omega_j,
\]
satisfy
\[
M_j(y_j)>2j\,\dist(y_j,\partial\Omega_j)^{-1}.
\]
By the doubling lemma
\cite[Lemma~5.1 and Remark~5.2(b)]{PQS2007}, there exists
$x_j\in\Omega_j$ such that
\[
M_j(x_j)>2j\,\dist(x_j,\partial\Omega_j)^{-1}
\]
and
\begin{equation}\label{eq:polacik-doubling-bound-hessian}
	M_j(z)\le 2M_j(x_j),
    \qquad
	|z-x_j|\le jM_j(x_j)^{-1}.
\end{equation}
We now rescale $u_j$ by
setting
\[
U_j(\xi):=\lambda_j^\alpha
u_j(x_j+\lambda_j\xi),
\qquad |\xi|\le j,
\]
where $\lambda_j:=M_j(x_j)^{-1}$.
Then
\[
U_j\in C^2(B_j)\cap\Phi^k(B_j),
\qquad U_j\le0.
\]
Since $k(\alpha+2)=n$ and $\alpha\frac{nk}{n-2k}=n$,
the rescaled function satisfies
\begin{equation}\label{eq:scaled-gradient-equation-new}
	F_k[U_j]=(-U_j)^{\frac{nk}{n-2k}}
	\qquad\text{in }B_j.
\end{equation}
Moreover,
\begin{equation}\label{eq:gradient-normalization-bound-new}
	(-U_j(0))^{\frac1\alpha}
	+|\nabla U_j(0)|^{\frac1{\alpha+1}}
	=1,
\end{equation}
whereas \eqref{eq:polacik-doubling-bound-hessian} yields
\[
(-U_j(\xi))^{\frac1\alpha}
+|\nabla U_j(\xi)|^{\frac1{\alpha+1}}
\le2,
\qquad |\xi|\le j.
\]
Consequently,
\begin{equation}\label{eq:uniform-value-gradient-rescaled-new}
	0\le -U_j(\xi)\le2^\alpha,
	\qquad
	|\nabla U_j(\xi)|\le2^{\alpha+1},
	\qquad |\xi|\le j.
\end{equation}

\medskip
\noindent\textbf{Step 2: compactness and the global limit.}
By
\eqref{eq:uniform-value-gradient-rescaled-new} and Arzel\`a--Ascoli,
we obtain a locally Lipschitz function $U\le0$ such that
\[
        U_j\to U
        \qquad\text{locally uniformly in }\Rn.
\]
In particular, $U_j\to U$ in $L^1_{\loc}$.  Since
$U_j\in\Phi^k(B_j)$, after taking $j\to\infty$, we obtain $U\in\Phi^k(\Rn)$.
Weak continuity of Hessian measures gives $F_k[U_j]\rightharpoonup F_k[U]$ on every compact subset of $\mathbb{R}^n$.
On the other hand, the local uniform convergence implies, for every
$\varphi\in C^\infty_c(\Rn)$,
\[
        \int \varphi\,(-U_j)^{\frac{nk}{n-2k}}\,d\xi
        \longrightarrow
        \int \varphi\,(-U)^{\frac{nk}{n-2k}}\,d\xi.
\]
Using \eqref{eq:scaled-gradient-equation-new} and comparing the two limits,
we obtain
\begin{equation}\label{eq:limit-equation-new}
        F_k[U]=(-U)^{\frac{nk}{n-2k}}
        \quad\text{in }\mathcal D'(\Rn).
\end{equation}

\medskip
\noindent\textbf{Step 3: identification of the global limit.}
The Liouville type theorem
\cite[Corollary~2.8]{PhucVerbitsky2008} rules out nontrivial entire solutions of \eqref{eq:limit-equation-new}.  Hence $U\equiv0$, and therefore $\|U_j\|_{L^\infty(B_1)}\to0$.

\medskip
\noindent\textbf{Step 4: exclusion of the gradient-dominated limit.}
The normalization and the strong maximum principle give $U_j<0$ in
$B_j$.  Hence the equation is locally uniformly elliptic along $U_j$,
and Evans--Krylov theory and Schauder bootstrapping yield
$U_j\in C^\infty_{\loc}(B_j)$.
Denote $F_k[U_j]=\psi_j^k$ and $\psi_j=(-U_j)^{\frac{n}{n-2k}}$.
Since $\frac{n}{n-2k}>1$, 
\begin{equation}\label{eq:psi-gradient-goes-zero-new}
        \|\nabla\psi_j\|_{L^\infty(B_1)}
        \le C\|U_j\|_{L^\infty(B_1)}^{\frac{2k}{n-2k}}
             \|\nabla U_j\|_{L^\infty(B_1)}
        \to0.
\end{equation}
Also $\osc_{B_1}U_j\le \|U_j\|_{L^\infty(B_1)}\to0$.  Applying
\eqref{eq:trudinger-gradient-original} with $R=1$ at the center $0$ we have 
\begin{equation}\label{eq:gradient-estimate-applied-at-origin}
        |\nabla U_j(0)|
        \le C(n,k)\left(
        \osc_{B_1}U_j+
        \|\nabla\psi_j\|_{L^\infty(B_1)}
        \right)
        \to0.
\end{equation}
On the other hand, \eqref{eq:gradient-normalization-bound-new} and the fact that
$U_j(0)\to0$ imply $|\nabla U_j(0)|\to1$, a contradiction.
This proves \eqref{eq:huang-xu-type-gradient-bound}.  The remaining estimate
\eqref{eq:serrin-fundamental-bound} follows by applying the estimate to
$\Omega=B_1\setminus\{0\}$, where
$\dist(x,\partial\Omega)=|x|$ for $0<|x|<\frac{1}{2}$.
\end{proof}

\begin{remark}\label{rem:super-serrin-apriori}
	Set $p_{\mathrm{Ser}}=\frac{nk}{n-2k}$ and $p_*=\frac{k(n+2)}{n-2k}$.
	Although only the case $p=p_{\mathrm{Ser}}$ is used below, the same
	doubling argument applies to every fixed
	$p_{\mathrm{Ser}}<p<p_*$.  More precisely, every nonpositive classical
	$k$-admissible solution of $F_k[u]=(-u)^p
	$ in $\Omega$
	satisfies
	\begin{equation}\label{eq:super-serrin-first-order-bound}
		|u(x)|^{\frac{p-k}{2k}}
		+|\nabla u(x)|^{\frac{p-k}{p+k}}
		\le C(n,k,p)\,\dist(x,\partial\Omega)^{-1}.
	\end{equation}
	Indeed, with
    $a_p=\frac{2k}{p-k}$
	one applies the doubling lemma to
	$|u|^{\frac1{a_p}}+|\nabla u|^{\frac1{a_p+1}}$
	and rescales by
    $U_j(\xi)=\lambda_j^{a_p}
    u_j(x_j+\lambda_j\xi)$.
	Since $k(a_p+2)=pa_p$, the equation is invariant under this
	rescaling.  If \eqref{eq:super-serrin-first-order-bound} failed, the
	compactness argument in the proof of
	Theorem~\ref{thm:serrin-fundamental-bound} would produce a nonpositive
	locally bounded entire solution $U$ of $F_k[U]=(-U)^p$. 
	The Liouville theorem of Dai, Fu, Gui and Qin
	\cite[Theorem~1.2]{DaiFuGuiQin2026} gives $U\equiv0$, and the final
	gradient argument in Step~4 then contradicts the normalization.
\end{remark}

\begin{remark}
\label{rem:conformal-vs-real-hessian}
Huang--Xu established an estimate of the form
\eqref{eq:huang-xu-type-gradient-bound} for positive conformally
$k$-admissible solutions of
\[
        F_k[A_w]=w^q,
        \qquad
        A_w=D^2w-\frac{|Dw|^2}{2w}I;
\]
see \cite[Theorem~1.1]{HuangXu2009}.  In that conformal setting, the
Liouville-type doubling argument admits a higher-order compactness approach, since
the interior estimate of Sheng, Trudinger and Wang
\cite[Theorem~3.1]{ShengTrudingerWang2007} provides the uniform $C^{3,\alpha}$ bounds.
\end{remark}

\begin{remark}
	\label{rem:wang-tian-second-order-compactness}
	Wang and Tian \cite[Theorem~2.1]{WangTian2020} obtained a related
	conditional estimate containing an additional Hessian term.  For $p>k$,
	their rescaling gives locally uniform $W^{2,\infty}$ bounds, but not the
	$C^{2,\alpha}_{\loc}$ compactness invoked to pass this term to the limit.
	The estimates cited from \cite[Section~3]{Wang2009KHessian} require uniform
	ellipticity, which is not ensured when the rescaled right-hand sides may
	tend to zero and the Hessians may approach $\partial\Gamma_k$.  Thus the
	Hessian-dominated case requires an additional argument.  This concerns only
	the proof as written and does not preclude another proof of the estimate.
	By contrast, our normalization involves only zeroth- and first-order terms,
	handled by local uniform convergence and Trudinger's gradient estimate,
	respectively.
\end{remark}

\subsection{Exclusion of a Dirac mass and a removability criterion}
\label{subsec:critical-removable-bootstrap}

We first record the finite Wolff-potential iteration used in the
removability argument.  For this purpose, set
\begin{equation}\label{eq:critical-bootstrap-exponents}
        Q:=\frac{n^2}{2(n-2k)}
        =\frac{nk}{n-2k}+\frac n2,
        \qquad
        T(q):=\frac{nq(n-2k)}{n^2-2q(n-2k)},
        \qquad \frac{nk}{n-2k}<q<Q.
\end{equation}
The iteration yields the implication
\[
        -u\in L^{\frac{nk}{n-2k}+\delta}
        \quad\Longrightarrow\quad
        -u\in L^a\text{ for some }a>Q
        \quad\Longrightarrow\quad
        (-u)^{\frac{nk}{n-2k}}\in L^m\text{ for some }m>\frac n{2k},
\]
and hence places the density in the range required by the H\"older
criterion in Proposition~\ref{thm:general-removable-lm}.

\begin{proposition}\label{prop:finite-crossing}
Let $u\le0$, $u\in\Phi^k(\Omega)$, and
$F_k[u]=(-u)^{\frac{nk}{n-2k}}$ in $\Omega$.  If
$u\in L^{\frac{nk}{n-2k}+\delta}_{\loc}(\Omega)$ for some $\delta>0$,
then $u\in L^a_{\loc}(\Omega)$ for some
$a>Q=\frac{n^2}{2(n-2k)}$.
\end{proposition}

\begin{proof}
	We omit the details here, since it is almost the same as in the
	proof of Lemma~\ref{lem:subcritical-wolff-crossing}.
\end{proof}

We next combine the measure decomposition with the lower Wolff estimate
to exclude a Dirac mass at the origin.

\begin{proposition}\label{prop:no-dirac-pv}
Under the assumptions of Theorem~\ref{thm:hessian-log-limit}, there holds
\[
        F_k[u]=(-u)^{\frac{nk}{n-2k}}
        \qquad\text{in }\mathcal D'(B_1).
\]
Moreover,
\begin{equation}\label{eq:k-zero-mass-asymptotic}
        \lim_{x\to0}|x|^{\frac{n-2k}{k}} u(x)=0.
\end{equation}
\end{proposition}

\begin{proof}
Set $q=\frac{nk}{n-2k}$ and $\eta=\frac{n-2k}{k}$, so that
$\eta q=n$.  By Lemma~\ref{lem:hessian-measure-decomposition-origin},
the canonical extension belongs to $\Phi^k(B_1)$ and satisfies
\[
        F_k[u]=M\delta_0+(-u)^q\,dx
        \quad\text{in }B_1,
        \qquad M\ge0,
\]
with $(-u)^q\in L^1_{\loc}(B_1)$.
Suppose that $M>0$.  For $0<r=|x|<\frac1{128}$, apply the lower
estimate in Lemma~\ref{lem:labutin-local-pointwise} with center $x$
and $R=\frac14$.  Since $B(x,\frac34)\subset B_1$,
$4r<\frac1{32}$, and $F_k[u](B(x,t))\ge M$ whenever $t>r$,
we obtain
\[
\begin{aligned}
        -u(x)
        &\ge C_1\W_k^{F_k[u]}(x,\frac1{32}) \\
        &\ge C_1M^{\frac1k}\int_{2r}^{4r}t^{-\eta}\frac{dt}{t}
        =cM^{\frac1k}r^{-\eta},
\end{aligned}
\]
where $c=c(n,k)>0$.  It follows that
\[
        (-u(x))^q\ge c^qM^{\frac qk}|x|^{-n},
        \qquad 0<|x|<\frac1{128},
\]
contradicting $(-u)^q\in L^1_{\loc}(B_1)$.  Therefore $M=0$,
which proves the asserted distributional identity.

It remains to prove \eqref{eq:k-zero-mass-asymptotic}.
By
Theorem~\ref{thm:serrin-fundamental-bound},
\[
|u(x)|\le C|x|^{-\frac{n-2k}{k}},
\qquad
|\nabla u(x)|\le C|x|^{-\frac{n-2k}{k}-1}
\qquad (0<|x|<\frac{1}{2}).
\]
If \eqref{eq:k-zero-mass-asymptotic} failed, there would be  $\varepsilon>0$ and $x_j\to0$ with
$|x_{j+1}|<\frac{|x_j|}{8}$ and
$|x_j|^{\frac{n-2k}{k}}(-u(x_j))\ge\varepsilon$.  On the other hand, for 
$y\in B(x_j,\frac{|x_j|}{4})$ the gradient bound gives
\[
|\nabla u(y)|\le C_0|x_j|^{-\frac{n-2k}{k}-1}.
\]
Choose $0<c<\frac{1}{4}$ so that $C_0c\le\frac{\varepsilon}{2}$.  Integrating along the
line segment from $x_j$ to $y$ then gives
\[
-u(y)\ge \frac{\varepsilon}{2}|x_j|^{-\frac{n-2k}{k}}
\qquad\text{for }y\in B(x_j,c|x_j|).
\]
After decreasing $c$ once more if necessary, these balls are pairwise
disjoint.  Therefore
\[
\int_{B_{\frac{1}{2}}}(-u)^{\frac{nk}{n-2k}}\,dx
\ge c_0\sum_j |x_j|^n|x_j|^{-\frac{n-2k}{k}\frac{nk}{n-2k}}=\infty,
\]
because $\frac{n-2k}{k}\frac{nk}{n-2k}=n$.  This contradicts
$(-u)^{\frac{nk}{n-2k}}\in L^1(B_{\frac{1}{2}})$.
\end{proof}

\begin{corollary}\label{cor:continuous-extension}
Assume the hypotheses of Proposition~\ref{prop:no-dirac-pv}.  If, for some
$\delta>0$,
\begin{equation}\label{eq:above-serrin}
        u\in L^{\frac{nk}{n-2k}+\delta}(B_1),
\end{equation}
then the singularity at the origin is removable.
\end{corollary}

\begin{proof}
Applying
Proposition~\ref{prop:finite-crossing} to the extended solution in a smaller ball,
we obtain an exponent $a>Q$ such that $u\in L^a(B_r)$
for some $r>0$.  Therefore
\[
        (-u)^{\frac{nk}{n-2k}}
        \in L^{\frac{a(n-2k)}{nk}}(B_r),
        \qquad
        \frac{a(n-2k)}{nk}
        >\frac{Q(n-2k)}{nk}=\frac n{2k}.
\]
By Proposition~\ref{thm:general-removable-lm}, $u$
is locally $C^{0,\gamma}$ at the origin for some $\gamma\in(0,1)$.
\end{proof}

\subsection{Annular Harnack estimates and asymptotic radiality}
\label{sec:annular-harnack}

The zero-mass decay obtained above allows the local Harnack inequality to
be iterated on fixed annuli with constants independent of the radius.  We
use the resulting annular estimate first to distinguish the removable case
from the case $u(x)\to-\infty$, and then to prove asymptotic radiality by
a rescaling argument.
For $0<\varepsilon<1$ and $r>0$, we set
\[
        \mathcal A_r^\varepsilon
        =\{x\in\mathbb R^n:\ \varepsilon r\le |x|\le
        \varepsilon^{-1}r\}.
\]

\begin{lemma}\label{lem:mohammed-porru-harnack}
Let $q>k$, $A>0$, and $0<\theta\le A$.  If $W\in C^3(B_1)$ is
negative and $k$-admissible, $0<-W\le1$, and
$F_k[W]=\theta(-W)^q$ in $B_1$, then
\[
        \sup_{B_{\frac{1}{2}}}(-W)
        \le C(n,k,q,A)\inf_{B_{\frac{1}{2}}}(-W).
\]
\end{lemma}

\begin{proof}
This follows directly from \cite[Theorems~3.2 and~4.4]{MohammedPorru2026}.  
\end{proof}

\begin{proposition}\label{prop:local-small-coefficient-k-hessian-harnack}
Let $B_{2R}(x_0)\Subset\Omega$, let $a>0$, and let
$U\in C^2(\Omega)$ be negative and $k$-admissible, with
$F_k[U]=a(-U)^{\frac{nk}{n-2k}}$ in $B_{2R}(x_0)$. Denote $S=\sup\limits_{B_{2R}(x_0)}(-U)$ and $ \Theta=aR^{2k}S^{\frac{2k^2}{n-2k}}$.
Then
\begin{equation}\label{eq:local-small-coeff-harnack}
        \sup_{B_R(x_0)}(-U)
        \le C(n,k,\Theta)\inf_{B_R(x_0)}(-U).
\end{equation}
\end{proposition}

\begin{proof}
Since $U<0$ and $a>0$, the equation is locally uniformly elliptic along
$U$.  Evans--Krylov theory and Schauder bootstrapping therefore give
$U\in C^\infty_{\loc}(B_{2R}(x_0))$.
For $|y|<1$, let  $\widetilde U(y)=S^{-1}U(x_0+2Ry)$.
Then $0<-\widetilde U\le1$ and
\[
        F_k[\widetilde U]
        =\theta(-\widetilde U)^{\frac{nk}{n-2k}},
        \qquad
        \theta=a(2R)^{2k}S^{\frac{2k^2}{n-2k}}=2^{2k}\Theta.
\]
Applying Lemma~\ref{lem:mohammed-porru-harnack} with exponent $\frac{nk}{n-2k}$ and
$A=2^{2k}\Theta$ gives
\[
        \sup_{B_{\frac{1}{2}}}(-\widetilde U)
        \le C(n,k,\Theta)
        \inf_{B_{\frac{1}{2}}}(-\widetilde U).
\]
Scaling back, we get \eqref{eq:local-small-coeff-harnack}.
\end{proof}

\begin{theorem}\label{thm:k-zero-mass-annular-harnack}
Under the assumptions of Theorem~\ref{thm:hessian-log-limit}, for every $0<\varepsilon<1$ there
exist constants $\delta=\delta(u,\varepsilon)>0$ and
$C=C(n,k,\varepsilon)>1$ such that
\begin{equation}\label{eq:k-annular-harnack}
        \sup_{\mathcal A_r^\varepsilon}(-u)
        \le C\inf_{\mathcal A_r^\varepsilon}(-u),
        \qquad 0<r<\delta.
\end{equation}
\end{theorem}

\begin{proof}
If $u\equiv0$, the conclusion is immediate.  Otherwise the strong maximum
principle gives $-u>0$ in the punctured ball.
Fix $0<\varepsilon<1$.  Choose finitely many balls $B(\xi_i,\rho_i)$, $i=1,\ldots,N$,
such that
\[
        \mathcal A_1^\varepsilon\subset\bigcup_{i=1}^N B(\xi_i,\rho_i),
        \qquad
        B(\xi_i,2\rho_i)\Subset\mathcal A_1^{\frac{\varepsilon}{2}}.
\]
For $0<r<\frac{\varepsilon}{2}$ and $U_r(\xi):=u(r\xi)$, we have 
\begin{equation}\label{eq:rescaled-k-harnack-equation}
        F_k[U_r]
        =r^{2k}F_k[u](r\xi)
        =r^{2k}(-U_r)^{\frac{nk}{n-2k}}.
\end{equation}
For each $i$,
Proposition~\ref{prop:local-small-coefficient-k-hessian-harnack}, applied to
$U_r$ on $B(\xi_i,2\rho_i)$, gives a Harnack constant depending only on
\[
        \Theta_i(r)
        =r^{2k}\rho_i^{2k}M_i(r)^{\frac{2k^2}{n-2k}},
\]
where $M_i(r):=\sup_{B(\xi_i,2\rho_i)}(-U_r)$. It suffices to bound $\Theta_i(r)$ uniformly as $r\to0$.

Since $\frac{n-2k}{k}\frac{2k^2}{n-2k}=2k$,
we have, for $\xi\in\mathcal A_1^{\frac{\varepsilon}{2}}$,
\begin{equation}\label{eq:zero-mass-controls-local-coefficient}
\begin{aligned}
        r^{2k}(-U_r(\xi))^{\frac{2k^2}{n-2k}}
        &=r^{2k}(-u(r\xi))^{\frac{2k^2}{n-2k}}                                      \\
        &=|\xi|^{-2k}
          \bigl(|r\xi|^{\frac{n-2k}{k}}(-u(r\xi))\bigr)^{\frac{2k^2}{n-2k}}\\
        &\le C(\varepsilon)
        \left(\sup_{y\in\mathcal A_r^{\frac{\varepsilon}{2}}}
        |y|^{\frac{n-2k}{k}}(-u(y))\right)^{\frac{2k^2}{n-2k}}.
\end{aligned}
\end{equation}
By \eqref{eq:k-zero-mass-asymptotic}, the last supremum tends to zero as
$r\to0$.  Hence one may choose
$\delta_0=\delta_0(u,\varepsilon)\le\frac{\varepsilon}{2}$ such that, for
$0<r<\delta_0$ and every $i$,
\begin{equation}\label{eq:local-coeff-small-on-chain}
        r^{2k}\rho_i^{2k}M_i(r)^{\frac{2k^2}{n-2k}}\le1.
\end{equation}
Applying
Proposition~\ref{prop:local-small-coefficient-k-hessian-harnack} on each
ball, using $\Theta_i(r)\le1$, gives
\begin{equation}\label{eq:local-chain-harnack-corrected}
        \sup_{B(\xi_i,\rho_i)}(-U_r)
        \le C_0(n,k,\varepsilon)\inf_{B(\xi_i,\rho_i)}(-U_r),
        \qquad i=1,\ldots,N,
\end{equation}
where $C_0$ is independent of $r$.  It is now easy to prove
\[
        \sup_{\mathcal A_1^\varepsilon}(-U_r)
        \le C(n,k,\varepsilon)\inf_{\mathcal A_1^\varepsilon}(-U_r),
        \qquad 0<r<\delta_0.
\]
Scaling back, we obtain \eqref{eq:k-annular-harnack}.
\end{proof}

The annular estimate separates the removable case from the case in which
the solution diverges to $-\infty$ at the origin.

\begin{lemma}
\label{lem:critical-bounded-or-blowup}
Under the assumptions of Theorem~\ref{thm:hessian-log-limit}, suppose that
$u\not\equiv0$.  Then either $u(x)\to-\infty$ as $x\to0$, or the
singularity is removable and $u(0)<0$.
\end{lemma}

\begin{proof}
Suppose that $u$ does not tend to $-\infty$ as $x\to0$.  Then there
are $M<\infty$ and a sequence $z_j\to0$ such that $-u(z_j)\le M$.
We prove that $u$ is bounded near the origin.

Suppose for contradiction that $-u$ is unbounded in every punctured
neighborhood of the origin.  Fix $\varepsilon_0=\frac{1}{4}$, and let $C_H$ and
$\delta_H$ be the constants in
Theorem~\ref{thm:k-zero-mass-annular-harnack} for this value of
$\varepsilon_0$.  Choose $y_j^+$ so that
\[
        0<|y_j^+|<\delta_H,
        \qquad
        -u(y_j^+)\ge jC_H.
\]
After passing to a subsequence of $z_j$, we can arrange that $|z_j|<\varepsilon_0^2|y_j^+|$ with $-u(z_j)\le M$.
Since $-u$ is unbounded in every punctured ball, choose $y_j^-$ with 
\[
        0<|y_j^-|<\varepsilon_0^2|z_j|,
        \qquad
        -u(y_j^-)\ge jC_H.
\]
Applying  Theorem~\ref{thm:k-zero-mass-annular-harnack} on
$\mathcal A_{|y_j^+|}^{\varepsilon_0}$ and
$\mathcal A_{|y_j^-|}^{\varepsilon_0}$, we obtain $-u\ge j$ on $\mathcal A_{|y_j^+|}^{\varepsilon_0}
\cup\mathcal A_{|y_j^-|}^{\varepsilon_0}$.

Since $-u\ge j$ on both boundary components of the annulus $A_j=\{\varepsilon_0^{-1}|y_j^-|<|x|<\varepsilon_0|y_j^+|\}$, the minimum principle yields
\[
        -u(z_j)\ge\min_{\partial A_j}(-u)\ge j,
\]
contradicting $-u(z_j)\le M$ for $j>M$.  Thus $u$ is bounded near
the origin.

In the bounded case, $(-u)^{\frac{nk}{n-2k}}\in L^m$ in a smaller ball for every
finite $m>1$.  Proposition~\ref{prop:no-dirac-pv} extends the equation
at the origin, and Proposition~\ref{thm:general-removable-lm} makes
the singularity removable.  Finally,
$u(0)<0$: if $u(0)=0$, the subharmonic function $u\le0$ would attain
an interior maximum at the origin, and the strong maximum principle would
give $u\equiv0$, contrary to the hypothesis.
\end{proof}

We now pass to normalized rescalings.  In the nonremovable case, the
annular Harnack estimate gives compactness on fixed annuli, and the global
B\^ocher theorem identifies every limiting function.

\begin{theorem}\label{thm:k-mazang-theorem51}
Under the assumptions of Theorem~\ref{thm:hessian-log-limit}, suppose that
$u\not\equiv0$.  With $\underline u$ and $\overline u$ defined as above,
\begin{equation}\label{eq:k-mazang-angular-asymptotic}
        \lim_{x\to0}\frac{-u(x)}{\underline{u}(|x|)}=1.
\end{equation}
In particular,
\begin{equation}\label{eq:k-mazang-spherical-ratio}
        \lim_{r\to0}\frac{\overline{u}(r)}{\underline{u}(r)}=1.
\end{equation}
\end{theorem}

\begin{proof}
If the singularity is removable, then $u(0)<0$ and
$\underline{u}(r)\to-u(0)$ and $u(x)\to u(0)$, so
\eqref{eq:k-mazang-angular-asymptotic} is immediate.  Hence it remains to
prove
\eqref{eq:k-mazang-angular-asymptotic} in the case in which $u(x)\to-\infty$ as $x\to0$.

Let $x_j\to0$ be arbitrary.  Let $r_j=|x_j|$ and $e_j=x_j/|x_j|\in \mathbb{S}^{n-1}$.
After passing to a subsequence, we may assume
$e_j\to e_\infty\in\mathbb S^{n-1}$.  Define $U_j(\xi)=u(r_j\xi)/\underline{u}(r_j)$.
Then $U_j\in\Phi^k$, $U_j<0$, and
\begin{equation}\label{eq:normalized-zero-mass-scaled-equation}
        F_k[U_j]
        =a_j(-U_j)^{\frac{nk}{n-2k}},
\end{equation}
where $a_j:=r_j^{2k}\underline{u}(r_j)^{\frac{2k^2}{n-2k}}\to 0$, since  $\frac{n-2k}{k}\frac{2k^2}{n-2k}=2k$ and by \eqref{eq:k-zero-mass-asymptotic}
\[
        r_j^{\frac{n-2k}{k}}\underline{u}(r_j)
        \le \sup_{|x|=r_j}|x|^{\frac{n-2k}{k}}(-u(x))
        \longrightarrow0.
\]
Fix $K\Subset\mathbb R^n\setminus\{0\}$, and choose
$0<\varepsilon<1$ such that
$
K\Subset\operatorname{int}\mathcal A_1^\varepsilon.
$
By Theorem~\ref{thm:k-zero-mass-annular-harnack}, for all sufficiently
large $j$,
\[
C(\varepsilon)^{-1}\le -U_j\le C(\varepsilon)
\qquad\text{on }\mathcal A_1^\varepsilon.
\]
Since $a_j\to0$, it follows that
\begin{equation}\label{eq:scaled-density-goes-zero}
	a_j(-U_j)^{\frac{nk}{n-2k}}\longrightarrow0
	\qquad\text{in }L^\infty(\mathcal A_1^\varepsilon).
\end{equation}

Fix $m_0>\frac{n}{2k}$.  The preceding bounds give uniform
$L^\infty$ bounds for $U_j$ and uniform $L^{m_0}$ bounds for
$F_k[U_j]$ locally in
$\operatorname{int}\mathcal A_1^\varepsilon$.  Applying
\eqref{eq:trudinger-holder-original} yields equicontinuity of $U_j$
on $K$.  Since $K\Subset\mathbb R^n\setminus\{0\}$ is arbitrary,
a diagonal argument gives, after passing to a subsequence,
\begin{equation}\label{eq:normalized-Uj-uniform-limit}
	U_j\longrightarrow U
	\qquad\text{locally uniformly in }\mathbb R^n\setminus\{0\}.
\end{equation}
Therefore, $U\in\Phi^k(\mathbb R^n\setminus\{0\})$ and $U\le0$.
Let $\varphi\in C_c^\infty(\mathbb R^n\setminus\{0\})$, and choose
$0<\varepsilon<1$ such that
$
\supp\varphi\Subset\operatorname{int}\mathcal A_1^\varepsilon.
$
By the weak continuity of Hessian measures, we have $F_k[U_j]\rightharpoonup F_k[U]$ on $\operatorname{int}\mathcal A_1^\varepsilon$,
whereas \eqref{eq:scaled-density-goes-zero} gives
\[
\int
\varphi\,a_j(-U_j)^{\frac{nk}{n-2k}}\,d\xi
\longrightarrow0.
\]
Combining this with
\eqref{eq:normalized-zero-mass-scaled-equation}, we obtain $F_k[U]=0$ in $\mathbb{R}^n\setminus\{0\}$.
Then it follows from Theorem~\ref{thm:k-hessian-punctured-bocher} that
\[
        U(\xi)=B-A|\xi|^{2-\frac{n}{k}},
        \qquad A\ge0,
        \quad B\le0.
\]
Notice that $\inf_{|\xi|=1}(-U_j(\xi))=1$
for every $j$. Thus the locally uniform convergence on $\mathbb S^{n-1}$
gives
\[
        \inf_{|\xi|=1}(-U(\xi))=1.
\]
The function $U$ is radial, so $-U\equiv1$ on $\mathbb S^{n-1}$.
Hence
\[
        \frac{-u(x_j)}{\underline{u}(|x_j|)}
        =-U_j(e_j)
        \longrightarrow -U(e_\infty)=1.
\]
Since $x_j\to0$ was arbitrary,
this proves \eqref{eq:k-mazang-angular-asymptotic}.
\end{proof}

\subsection{Sharp logarithmic asymptotics}
\label{sec:critical-log}

Having established asymptotic radiality, it remains to determine the
exact logarithmic behavior.  We use the Emden--Fowler transformation and
one-sided measure differential inequalities for the spherical extrema:
the spherical minimum gives the upper bound, while the spherical maximum
gives the matching lower bound.

We first record the elementary facts on semiconvexity and semiconcavity
used below.

\begin{definition}[Semiconvexity and semiconcavity]
Let $f:(\alpha,\beta)\to\mathbb R$.
\begin{enumerate}[label=\textup{(\roman*)}]
\item The function $f$ is \emph{semiconvex} if there exists $C>0$ such that
      $f(r)+Cr^2$ is convex on $(\alpha,\beta)$.
\item The function $f$ is \emph{semiconcave} if there exists $C>0$ such that
$f(r)-Cr^2$ is concave on $(\alpha,\beta)$.
\end{enumerate}
\end{definition}

\begin{lemma}\label{lem:semiconcavity-properties}
Let $f:(\alpha,\beta)\to\mathbb R$ be locally semiconcave.  Then $f$ is
locally Lipschitz, the one-sided derivatives $f_r^\pm$ exist everywhere,
and $f_r^-\ge f_r^+$.  Both derivatives are locally of bounded variation;
$f_r^-$ is left-continuous, whereas $f_r^+$ is right-continuous and lower
semicontinuous.  Moreover,
\[
        f_{rr}=\nu_{ac}+\nu_{pp}+\nu_{sc},
        \qquad \nu_{pp},\nu_{sc}\le0,
\]
the set $\{r:f_r^-(r)>f_r^+(r)\}$ is at most countable, and
\[
        f_r^-(\beta_1)-f_r^+(\alpha_1)
        =f_{rr}((\alpha_1,\beta_1))
        \quad\text{for }[\alpha_1,\beta_1]\Subset(\alpha,\beta).
\]
For locally semiconvex $f$, the corresponding assertions hold with the
inequalities reversed and $f_r^+$ upper semicontinuous.
\end{lemma}

\begin{proof}
On a compact subinterval, subtracting or adding a quadratic polynomial
reduces the assertion to the standard one-dimensional theory of convex and
concave functions.  The sign of the singular part of the second derivative
is inherited from convexity or concavity, while the last identity is the
Newton--Leibniz formula for the bounded-variation representative of $f'$.
For a semiconcave function, $f_r^+$ differs by a continuous affine
function from a right-continuous nonincreasing function and is therefore
lower semicontinuous; the semiconvex case is analogous with the inequalities
reversed.
\end{proof}

We first apply these facts to the spherical minimum and derive the upper
bound.

\begin{proposition}
    \label{prop:hessian-log-upper-bound}
    Under the hypotheses of Theorem~\ref{thm:hessian-log-limit},
    \begin{equation}\label{eq:hessian-log-upper-bound}
        \limsup_{x\to0}
        \left(\log\frac1{|x|}\right)^{\frac{n-2k}{2k^2}}|x|^{\frac{n-2k}{k}}(-u(x))
        \le \mathcal L_{n,k}.
    \end{equation}
\end{proposition}
If $u\equiv0$, the proposition is immediate.  In the phase-plane analysis
below we therefore assume $u\not\equiv0$; the strong maximum principle gives
$u<0$ in the punctured ball.
Set
\[
        \eta:=\frac{n-2k}{k},
        \qquad
        t=\eta\log\frac{\eta}{r},
        \qquad 0<r<1,
\]
and define $\underline w$ by
$ e^t\underline w(t)=\underline{u}(r)$.
We also denote
\[
        \phi(r):=\max_{|x|=r}u(x)=-\underline{u}(r).
\]
The next two lemmas establish the one-sided measure inequality satisfied
by $\underline w$.

\begin{lemma}\label{lem:critical-w-semiconcavity}
The function $\underline w$ is locally semiconcave on
$(\eta\log\eta,\infty)$ and satisfies
\[
        \lim_{t\to\infty}\underline w(t)=0.
\]
If
\[
        \underline w_{tt}=\nu_{ac}+\nu_{pp}+\nu_{sc},
\]
then $\nu_{pp}$ and $\nu_{sc}$ are nonpositive, and $\nu_{pp}$ is supported
on the at most countable set
\[
        \mathcal J
        :=\{t:\underline w_t^-(t)>\underline w_t^+(t)\}.
\]
\end{lemma}

\begin{proof}
Fix a compact interval $I\Subset(\eta\log\eta,\infty)$ and write
$r(t)=\eta e^{-\frac{t}{\eta}}$.  The radii $r(t)$, $t\in I$, remain in a compact
subinterval of $(0,1)$.  For $\theta\in\mathbb S^{n-1}$ define
\[
        g_\theta(t)=-e^{-t}u(r(t)\theta).
\]
Since $u\in C^2$ away from the origin, the second derivatives
$g_\theta''$ are bounded above uniformly for $(t,\theta)\in
I\times\mathbb S^{n-1}$.  Thus, for some $C_I>0$, every function
$t\mapsto g_\theta(t)-C_It^2$ is concave on $I$.  Moreover, $\underline w(t)=\inf_{\theta\in\mathbb S^{n-1}}g_\theta(t)$  and
$-\underline w(t)+C_It^2=\sup_{\theta}(-g_\theta(t)+C_It^2)$ is convex.  Hence
$\underline w(t)-C_It^2$ is concave on $I$, and $\underline w$ is locally
semiconcave.

By Proposition~\ref{prop:no-dirac-pv},
$ \underline w(t)
=\eta^{-\eta}r^\eta\underline{u}(r)\to0$ as $t\to\infty$.
The assertions concerning the distributional second derivative now follow
from Lemma~\ref{lem:semiconcavity-properties}.
\end{proof}

From now on we denote
\[
        \underline z(t)=\underline w_t^+(t),
        \qquad q=\frac{nk}{n-2k},
        \qquad
        A_k=\frac{\binom{n-1}{k-1}}{\eta^{k-1}}.
\]

\begin{lemma}\label{lem:critical-phase-properties}
	The phase-plane trajectory $(\underline w,\underline z)$ has the following
	properties.
	\begin{enumerate}[label=\textup{(\roman*)}]
		\item $\underline z$ is right-continuous, lower semicontinuous, and locally
		of bounded variation; it is continuous on
		$(\eta\log\eta,\infty)\setminus\mathcal J$.
		
		\item If $\underline w+\underline z>0$ on $[t_0,t_1]$, then
		\begin{align}
			\underline z(t_1)-\underline z(t_0)
			&\le \underline w_t^-(t_1)-\underline w_t^+(t_0) \notag\\
			&\le \int_{t_0}^{t_1}
			\left(
			-\underline z
			-\frac{\underline w^q}
			{A_k(\underline w+\underline z)^{k-1}}
			\right)\,dt.
			\label{eq:critical-phase-inequality}
		\end{align}
	\end{enumerate}
\end{lemma}

\begin{proof}
	Part \textup{(i)} follows from
	Lemmas~\ref{lem:semiconcavity-properties}
	and~\ref{lem:critical-w-semiconcavity}.  We prove \textup{(ii)}.
	
	Fix a point $t\in(t_0,t_1)$ at which $\underline w$ is twice
	differentiable, and set
	$
	r=\eta e^{-\frac{t}{\eta}}.
	$
	Then $\phi(r)=-\underline u(r)$ is twice differentiable at $r$.
	Choose $x_r\in\partial B_r$ such that
	$
	u(x_r)=\phi(r).
	$
	The radial function
	$
	\Phi(x):=\phi(|x|)
	$
	touches $u$ from above at $x_r$.  Hence
	$
	D^2\Phi(x_r)\ge D^2u(x_r).
	$
	Since $u$ is $k$-admissible, it follows that
	$D^2\Phi(x_r)\in\Gamma_k$; therefore, by the monotonicity of $F_k$
	on $\Gamma_k$,
	\[
	F_k[\Phi](x_r)
	\ge F_k[u](x_r)
	=(-u(x_r))^q
	=\underline u(r)^q.
	\]
	
	At almost every $t\in(t_0,t_1)$, one has
	$\underline w_t=\underline z$ and $	\phi'(r)
	=\frac{\eta e^t}{r}
	\bigl(\underline w+\underline w_t\bigr)$.  Since
	$\underline w+\underline z>0$, we have
    $\phi'(r)>0$.  Using the radial formula
	\[
	F_k[\Phi]
	=\frac{1}{k}\binom{n-1}{k-1}r^{1-n}
	\left(r^{n-k}(\phi')^k\right)'
	\]
	and the change of variables
	$t=\eta\log(\frac{\eta}{r})$, we obtain
	\[
	F_k[\Phi]
	=
	-A_ke^{qt}
	\bigl(\underline w+\underline w_t\bigr)^{k-1}
	\bigl(\underline w_t+\underline w_{tt}\bigr).
	\]
	Since $\underline u(r)=e^t\underline w(t)$,
	at almost every $t\in(t_0,t_1)$,
	we have \[
	f_{ac}(t)
	\le
	-\underline z(t)
	-\frac{\underline w(t)^q}
	{A_k\bigl(\underline w(t)+\underline z(t)\bigr)^{k-1}},
	\]
	where
	$
	\nu_{ac}=f_{ac}(t)\,dt
	$
	is the absolutely continuous part of $\underline w_{tt}$.
	
	By Lemma~\ref{lem:critical-w-semiconcavity}, the singular parts
	$\nu_{pp}$ and $\nu_{sc}$ are nonpositive.  Consequently,
	\[
	\underline w_{tt}
	\le
	\left(
	-\underline z
	-\frac{\underline w^q}
	{A_k(\underline w+\underline z)^{k-1}}
	\right)\,dt
	\]
	as an inequality of Radon measures on $(t_0,t_1)$.
	The generalized Newton--Leibniz formula therefore yields
	\begin{align*}
		\underline w_t^-(t_1)-\underline w_t^+(t_0)
		&=\underline w_{tt}\bigl((t_0,t_1)\bigr)\\
		&\le
		\int_{t_0}^{t_1}
		\left(
		-\underline z
		-\frac{\underline w^q}
		{A_k(\underline w+\underline z)^{k-1}}
		\right)\,dt.
	\end{align*}
	Finally,
	\[
	\underline z(t_1)
	=\underline w_t^+(t_1)
	\le \underline w_t^-(t_1),
	\qquad
	\underline z(t_0)=\underline w_t^+(t_0),
	\]
	which gives the first inequality in
	\eqref{eq:critical-phase-inequality}.
\end{proof}

The first consequence of the measure inequality is the eventual
negativity of $\underline z$.

\begin{lemma}\label{lem:critical-z-negative}
There exists $T>0$ such that
\[
        \underline z(t)<0
        \qquad\text{for all }t\ge T.
\]
\end{lemma}

\begin{proof}
There is arbitrarily large $t$ at which $\underline z<0$; otherwise
$\underline z\ge0$ eventually, so $\underline w$ would be nondecreasing
near infinity, contradicting $\underline w>0$ and
$\underline w(t)\to0$.  Choose $T>\eta\log\eta$ with
$\underline z(T)<0$ and suppose that $\underline{z}<0$ fails at a later time.  Set
\[
        t_1=\sup\bigl\{s\ge T:
        \underline z(t)<0\text{ for every }t\in[T,s)\bigr\}.
\]
By Part\textup{(i)} of Lemma~\ref{lem:critical-phase-properties}, we have $\underline{z}(t_1)=0$ and $\lim\limits_{t\to t^-_1}\underline z(t)=0$.

Since $\underline w(t_1)>0$, the preceding limits allow us to choose
$\delta>0$ such that $\underline w+\underline z>0$ and
\[
        -\underline z
        -\frac{\underline w^q}
        {A_k(\underline w+\underline z)^{k-1}}<0
        \qquad\text{on }[t_1-\delta,t_1].
\]
Applying \eqref{eq:critical-phase-inequality} on this interval gives
$\underline z(t_1)<\underline z(t_1-\delta)$, contradicting
$\underline z(t_1)=0$ and $\underline z<0$ to the left of $t_1$.
\end{proof}

 We next compare $\underline z$ with the barrier
 $-A\underline w^{q-k+1}$ and introduce the corresponding invariant
 region. For fixed $\varepsilon_0,\delta_0\in(0,1)$, we denote $ A:=A(\varepsilon_0)=\frac1{(1+\varepsilon_0)A_k}$ and 
\begin{equation}\label{eq:critical-invariant-region-def}
        \mathcal R_A^{\delta_0}
        :=\left\{(w,z):
        0<w\le\delta_0,
        \ z<-Aw^{q-k+1}
        \right\}.
\end{equation}

\begin{lemma}\label{lem:critical-region-implies-bound}
Suppose that there exists $T>0$ such that
\[
        (\underline w(t),\underline z(t))\in\mathcal R_A^{\delta_0}
        \qquad\text{for every }t>T.
\]
Then
\begin{equation}\label{eq:critical-epsilon-log-bound}
        \limsup_{x\to0}
        \left(\log\frac1{|x|}\right)^{\frac{n-2k}{2k^2}}
        |x|^{\frac{n-2k}{k}}(-u(x))
        \le (1+\varepsilon_0)^{\frac{1}{q-k}}\mathcal L_{n,k}.
\end{equation}
\end{lemma}

\begin{proof}
A locally semiconcave function is locally Lipschitz and hence locally
absolutely continuous.  Moreover,
$\underline w_t=\underline z$ almost everywhere.  Thus the defining
inequality of $\mathcal R_A^{\delta_0}$ gives
\[
        \frac{d}{dt}\underline w
        \le -A\underline w^{q-k+1}
        \qquad\text{for almost every }t>T.
\]
Thus,
\[
        \frac{d}{dt}\underline w^{k-q}\ge A(q-k).
\]
After integration from a fixed $t_1>T$ to $t>t_1$,
\[
        \underline w(t)^{k-q}
        \ge \underline w(t_1)^{k-q}+A(q-k)(t-t_1).
\]
Since $A=((1+\varepsilon_0)A_k)^{-1}$, it follows directly that
\begin{equation}\label{eq:critical-w-limsup}
        \limsup_{t\to\infty}t^{\frac{1}{q-k}}\underline w(t)
        \le
        \left(\frac{q-k}{(1+\varepsilon_0)A_k}\right)^{-\frac{1}{q-k}}.
\end{equation}
Since $r^{\frac{n-2k}{k}}\underline{u}(r)
=\eta^{\frac{n-2k}{k}}\underline w(t)$, transforming 
the above upper bound back to $\underline{u}$, we obtain
\[
        \limsup_{r\to0}
        \left(\log\frac1r\right)^{\frac{n-2k}{2k^2}}
        r^{\frac{n-2k}{k}}\underline{u}(r)
        \le (1+\varepsilon_0)^{\frac{1}{q-k}}\mathcal L_{n,k}.
\]
By \eqref{eq:k-mazang-angular-asymptotic}, the same upper bound holds for $u$ as well.
\end{proof}

The next lemma shows that the region $\mathcal R_A^{\delta_0}$
is forward invariant once it is reached for $t$ large enough.

In the first-contact arguments below and in Appendix~\ref{app:critical-lower-bound},
set $d=z-h(w)$.  Right-continuity and the one-sided jump direction give
$0\le d(t_1)\le d(t_1^-)\le0$ for an exit from $d<0$ with only downward
jumps, and the reversed inequalities for an exit from $d>0$ with only
upward jumps.  Hence $d(t_1^-)=d(t_1)=0$, so $z$ is continuous at the
first exit time $t_1$.

\begin{lemma}\label{lem:critical-invariant-region}
There exists $T>0$ such that, if
\[
        (\underline w(t_0),\underline z(t_0))
        \in\mathcal R_A^{\delta_0}
\]
for some $t_0>T$, then
\[
        (\underline w(t),\underline z(t))
        \in\mathcal R_A^{\delta_0}
        \qquad\text{for every }t>t_0.
\]
\end{lemma}

\begin{proof}
	Set
	\[
	h(w):=-Aw^{q-k+1}.
	\]
	Since $\underline{w}(t)\to 0$, we may first choose $T$ so large that $0<\underline w(t)<\delta_0$,
	$\underline z(t)<0$, and
	$A\underline w(t)^{q-k}<\frac{1}{2}$ for $t\ge T$.  Suppose for contradiction that the trajectory
	enters $\mathcal R_A^{\delta_0}$ at $t_0\ge T$ and later exits.  By Part\textup{(i)} of Lemma~\ref{lem:critical-phase-properties}, there is  a first contact time $t_1>t_0$ such that
	\[
	\underline z(t_1)=h(\underline w(t_1)),
	\qquad
	\underline z(t)<h(\underline w(t))\quad(t_0\le t<t_1).
	\]
	Moreover,
	\[
	\lim_{t\to t_1}\underline z(t)
	=\underline z(t_1)=h(\underline w(t_1)).
	\]
	Because $\underline z<0$, the function $\underline w$ is strictly decreasing.
	For $t<t_1$ close to $t_1$ we consequently have
	\begin{equation}\label{eq:critical-first-contact-upper}
		\begin{aligned}
			\frac{\underline z(t)-\underline z(t_1)}
			{\underline w(t)-\underline w(t_1)}
			&\le
			\frac{h(\underline w(t))-h(\underline w(t_1))}
			{\underline w(t)-\underline w(t_1)},
		\end{aligned}
	\end{equation}
	and the right-hand side tends to
	$h'(\underline w(t_1))<0$ as $t\to t^-_1$. On the other hand, at $t=t_1$,
	\[
	\underline w(t_1)+\underline z(t_1)
	=\underline w(t_1)
	\bigl(1-A\underline w(t_1)^{q-k}\bigr)>0.
	\]
	The same is true on a left neighborhood of $t_1$.  Applying
	Lemma~\ref{lem:critical-phase-properties} on $[t,t_1]$, dividing by the
	negative quantity
	$\underline w(t_1)-\underline w(t)=\int_t^{t_1}\underline z(s)\,ds$,
	and then letting $t\to t^-_1$, we obtain
	\begin{align*}
		\liminf_{t\to t^-_1}
		\frac{\underline z(t)-\underline z(t_1)}
		{\underline w(t)-\underline w(t_1)}
		&\ge
		-1+
		\frac{1}
		{A_kA\bigl(1-A\underline w(t_1)^{q-k}\bigr)^{k-1}}\\
		&\ge -1+\frac1{A_kA}=\varepsilon_0>0,
	\end{align*}
	which contradicts \eqref{eq:critical-first-contact-upper}.
\end{proof}

It remains to show that the trajectory eventually enters the invariant
region.

\begin{lemma}\label{lem:critical-eventual-entry}
Let $T$ be as in Lemma~\ref{lem:critical-invariant-region}.  Then there exists
$t_0\ge T$ such that
\[
        (\underline w(t_0),\underline z(t_0))
        \in\mathcal R_A^{\delta_0}.
\]
\end{lemma}

\begin{proof}
Suppose to the contrary that
\begin{equation}\label{eq:critical-outside-invariant-region}
        \underline z(t)
        \ge-A\underline w(t)^{q-k+1}
        \qquad\text{for every }t\ge T.
\end{equation}
After increasing $T$, Lemma~\ref{lem:critical-z-negative} gives
$\underline z<0$ on $[T,\infty)$, and we may also assume
$A\underline w(t)^{q-k}<\frac{1}{2}$.  Hence
\[
        \underline w+\underline z
        \ge \underline w\bigl(1-A\underline w^{q-k}\bigr)>0.
\]
Combining
\[
        -A\underline w(t)^{q-k+1}
        \le\underline z(t)<0,
\]
and the fact that $\underline w(t)\to0$, we obtain $\underline z(t)\to0$.
Set
\[
        G(t)=-\underline z(t)
        -\frac{\underline w(t)^q}
        {A_k(\underline w(t)+\underline z(t))^{k-1}}.
\]
Since $-\underline z\le A\underline w^{q-k+1}$ and
$\underline w+\underline z\le\underline w$, we have
\begin{align*}
        \frac{G(t)}{\underline z(t)}
        &=-1+
        \frac{\underline w(t)^q}
        {A_k(\underline w(t)+\underline z(t))^{k-1}
        (-\underline z(t))}\\
        &\ge -1+\frac1{A_kA}
        =\varepsilon_0.
\end{align*}
Because $\underline z<0$, this is equivalent to
$G\le\varepsilon_0\underline z$.  The measure form of
\eqref{eq:critical-phase-inequality} therefore yields
\[
        d\underline z
        \le G(t)\,dt
        \le \varepsilon_0\underline z(t)\,dt
        =\varepsilon_0\,d\underline w
\]
on $[T,\infty)$.  Thus
$\underline z(t)-\varepsilon_0\underline w(t)$ is nonincreasing and tends to zero as $t\to\infty$.
This contradicts $\underline z(t)<0<\underline w(t)$ and proves the lemma.
\end{proof}

\minisection{Completion of the upper bound}

\begin{proof}[Proof of Proposition~\ref{prop:hessian-log-upper-bound}]
By Proposition~\ref{prop:no-dirac-pv}, the zero-mass decay estimate
\eqref{eq:k-zero-mass-asymptotic} holds.  If $u$ is bounded near the origin,
then the left-hand side of \eqref{eq:hessian-log-upper-bound} is zero and
the conclusion is immediate.  Otherwise,
Lemma~\ref{lem:critical-bounded-or-blowup} gives
$u(x)\to-\infty$, and Theorem~\ref{thm:k-mazang-theorem51} supplies the
angular asymptotic needed to pass from spherical infima to pointwise
estimates.
Choose arbitrary $\varepsilon_0,\delta_0\in(0,1)$.  Lemma~\ref{lem:critical-z-negative}
gives $\underline z<0$ for large $t$, and Lemma~\ref{lem:critical-eventual-entry} provides
$t_0$ with
$(\underline w(t_0),\underline z(t_0))\in\mathcal R_A^{\delta_0}$.
Lemma~\ref{lem:critical-invariant-region} keeps the trajectory in
$\mathcal R_A^{\delta_0}$ for all later times.
Lemma~\ref{lem:critical-region-implies-bound}
therefore gives \eqref{eq:critical-epsilon-log-bound}.  Since
$\varepsilon_0>0$ was arbitrary, letting $\varepsilon_0\to0$ proves
\eqref{eq:hessian-log-upper-bound}.
\end{proof}

\minisection{The complementary lower bound}

The lower bound is obtained from the spherical maximum.  The corresponding
measure inequality and invariant-region argument require, respectively,
the verification of $k$-admissibility at lower contact points and the
exclusion of exponential decay.  These details are given in
Appendix~\ref{app:critical-lower-bound}.

\begin{proposition}
\label{prop:hessian-log-lower-bound}
Under the hypotheses of Theorem~\ref{thm:hessian-log-limit}, assume in
addition that $u(x)\to-\infty$ as $x\to0$.  Then
\begin{equation}\label{eq:hessian-log-lower-bound}
        \liminf_{x\to0}
        \left(\log\frac1{|x|}\right)^{\frac{n-2k}{2k^2}}|x|^{\frac{n-2k}{k}}(-u(x))
        \ge \mathcal L_{n,k}.
\end{equation}
\end{proposition}

\begin{proof}
Retain $\eta,q,A_k$, and $t$ from the proof of
Proposition~\ref{prop:hessian-log-upper-bound}, and set
\[
        \beta=q-k,
        \qquad
        e^t\overline w(t)=\overline u(r),
        \qquad
        \overline z(t)=\overline w_t^+(t).
\]
Fix $\varepsilon\in(0,1)$ and put
\[
        B_\varepsilon=\frac1{(1-\varepsilon)A_k}.
\]
Lemmas~\ref{lem:critical-bar-w-semiconvexity}--\ref{lem:critical-bar-eventual-entry}
in Appendix~\ref{app:critical-lower-bound} show that there are
$\delta_\varepsilon>0$ and $T_\varepsilon>0$ such that
\begin{equation}\label{eq:critical-bar-eventual-trapping}
        0<\overline w(t)\le\delta_\varepsilon,
        \qquad
        \overline z(t)>-B_\varepsilon\overline w(t)^{\beta+1}
        \qquad(t\ge T_\varepsilon).
\end{equation}
Since $\overline w$ is locally semiconvex, it is locally absolutely
continuous and $\overline w_t=\overline z$ almost everywhere.  Hence
\eqref{eq:critical-bar-eventual-trapping} gives
\[
        \frac d{dt}\overline w^{-\beta}
        =-\beta\overline w^{-\beta-1}\overline w_t
        \le \beta B_\varepsilon
        \qquad\text{a.e. on }(T_\varepsilon,\infty).
\]
After integration,
\[
        \liminf_{t\to\infty}
        t^{\frac1\beta}\overline w(t)
        \ge(\beta B_\varepsilon)^{-\frac1\beta}.
\]
Using $r^\eta\overline u(r)=\eta^\eta\overline w(t)$,
$t=\eta\log(\eta/r)$, and the identity defining
$\mathcal L_{n,k}$, we obtain
\[
        \liminf_{r\to0}
        \left(\log\frac1r\right)^{\frac1\beta}
        r^\eta\overline u(r)
        \ge(1-\varepsilon)^{\frac1\beta}\mathcal L_{n,k}.
\]
By Theorem~\ref{thm:k-mazang-theorem51},
$\overline u(r)/\underline u(r)\to1$.  Since
$-u(x)\ge\underline u(|x|)$, the same lower bound holds pointwise.
Finally, let $\varepsilon\to0$ and use
$\frac1\beta=\frac{n-2k}{2k^2}$ to obtain
\eqref{eq:hessian-log-lower-bound}.
\end{proof}

The upper and lower estimates now determine the exact asymptotic
constant.
\begin{proof}[Proof of Theorem~\ref{thm:hessian-log-limit}]
Lemma~\ref{lem:critical-bounded-or-blowup} gives the
dichotomy between removability and divergence.  In the removable case there
is nothing further to prove.  In the nonremovable case,
Proposition~\ref{prop:hessian-log-upper-bound} and
Proposition~\ref{prop:hessian-log-lower-bound} give matching bounds with the
same constant $\mathcal L_{n,k}$, proving the asserted limit
\eqref{eq:hessian-log-limit}.
Since $\frac{n-2k}{2k^2}>0$, the same limit also gives
$
        |x|^{\frac{n-2k}{k}}u(x)\to0,
$
which is \eqref{eq:hessian-subfundamental-limit}.
\end{proof}

\section{The critical dimension and gradient--exponential sources}
\label{sec:critical-dimension}

Throughout this section, $k=\frac{n}{2}$
and $u$ and $f$ satisfy the hypotheses of
Theorem~\ref{thm:critical-dimension-classification}.  

Let $\mu$ be the measure on $B_1$ defined by
$
        d\mu=f(x,u(x),\nabla u(x))\,dx.
$
  By Lemma~\ref{lem:hessian-measure-decomposition-origin}, the canonical
extension of $u$ to $B_1$ belongs to $\Phi^k(B_1)$ and, for some
$M_0\ge0$,
\begin{equation}\label{eq:critdim-measure-decomposition}
        F_k[u]=M_0\delta_0+f(x,u(x),\nabla u(x)).
\end{equation}

The proof adapts the strategy of Ma and Qing for isolated singularities
of $n$-superharmonic functions \cite{MaQing2021} to the Hessian
setting.  The Adams--Moser--Trudinger estimate below yields the
logarithmic growth bound, while blow-down compactness and a local
Wolff-potential estimate determine the asymptotic coefficient.

\subsection{An Adams--Moser--Trudinger estimate}
The following integrability estimate for $\W_k$ is used for controlling the gradient--exponential density.

\begin{proposition}[Adams--Moser--Trudinger inequality]
\label{prop:critdim-hessian-amt}
Let $k=\frac{n}{2}$, let $\Omega\subset\mathbb R^n$ be bounded with diameter $D$, and
let $f\in L^1(\Omega)$, $f\ge0$.  Then, for every $0<\delta<1$,
\begin{equation}\label{eq:critdim-hessian-amt}
\int_\Omega
\exp\left(
        \frac{n(1-\delta)\W_k^f(x,D)}
        {\|f\|_{L^1(\Omega)}^{\frac{1}{k}}}
\right)\,dx
\le
\frac{C(n)|B(0,D)|}{\delta^{n+1}}+2^n|\Omega|.
\end{equation}
When $f=0$ a.e., $\W_k^f\equiv0$, and the quotient in the exponent
is defined to be zero.
\end{proposition}

\begin{proof}
This follows from \cite[Lemma~2.2]{MaWang2025} with $N=n$, $p=k+1$, and
$\alpha=\frac{n}{k+1}$.
\end{proof}

\subsection{A uniform logarithmic upper bound}

\begin{lemma}
\label{lem:critdim-exp-quotient-bound}
There is a constant $C>0$ such that
\begin{equation}\label{eq:critdim-exp-quotient-bound}
        -u(x)\le C\log\frac1{|x|}+C,
        \qquad 0<|x|<\frac12.
\end{equation}
\end{lemma}

The proof also requires a local infimum estimate on balls whose radii
are comparable to their distance from the origin.

\begin{lemma}
\label{lem:critdim-infimum-estimate}
For every fixed $\theta\in(0,\frac{1}{2}]$, there is a constant
\[
        C_\theta=C\bigl(n,\theta,
        \|u\|_{L^n(B_1)}\bigr)>0
\]
such that
\begin{equation}\label{eq:critdim-infimum-estimate}
        \inf_{B(x_0,\theta|x_0|)}(-u)
        \le C_\theta\log\frac1{|x_0|},
        \qquad x_0\in B_{\frac{1}{2}}\setminus\{0\}.
\end{equation}
\end{lemma}

\begin{proof}
Set $w=-u\ge0$.  Since $u$ is $k$-admissible,
\[
        -\lambda(D^2w)=\lambda(D^2u)\in\overline\Gamma_k
        \qquad\text{in }B_1\setminus\{0\}.
\]
The cone comparison criterion in
\cite{LiuMaQingZhong2026} (see Lemmas~3.2.1--3.2.2 and Corollary~3.2.2) shows that $w$ is
$p_{\overline{\Gamma}_k}$-superharmonic, with
\begin{equation}\label{eq:critical-p-gamma-k}
        p_{\overline{\Gamma}_k}
        =\frac{n(k-1)}{n-k}+2.
\end{equation}
When $k=\frac{n}{2}$, this gives $p_{\overline{\Gamma}_k}=n$.  The standard
extension of $w$ across the zero $n$-capacity point $\{0\}$ is therefore
$n$-superharmonic in $B_1$
\cite[Theorem~7.35]{HeinonenKilpelainenMartio1993}; moreover
$w\in L^n(B_1)$.  Thus
\eqref{eq:critdim-infimum-estimate} for $\theta=\frac{1}{2}$ is precisely
\cite[Lemma~3.3]{MaQing2021}.  The same proof gives the stated estimate for
every fixed $\theta\in(0,\frac{1}{2}]$, with a constant depending also on
$\theta$.
\end{proof}

\begin{proof}[Proof of Lemma~\ref{lem:critdim-exp-quotient-bound}]
We follow the proof of \cite[Lemma~4.1]{MaQing2021}.  Suppose to the
contrary that there are $x_j\to0$ such that, with $\lambda_j=\log\frac1{|x_j|}\to\infty$,
one has $u(x_j)/\lambda_j\to-\infty$.

Let $a_j=\frac{|x_j|}{4}$ and $u_j(\xi)=u\left(x_j+a_j\xi\right)$ for $\xi\in B_2$.
Set
\begin{equation}\label{eq:critdim-local-rescaled-equation}
        g_j(\xi)
        :=a_j^n f\left(x_j+a_j\xi,u_j(\xi),a_j^{-1}\nabla u_j(\xi)\right)
        =F_k[u_j](\xi)
        \qquad (\xi\in B_2),
\end{equation}
and let $\mu_j$ be the measure on $B_2$ defined by
$
        d\mu_j=g_j(\xi)\,d\xi.
$
By Lemma~\ref{lem:labutin-local-pointwise} and Lemma~\ref{lem:critdim-infimum-estimate}, we obtain
\begin{equation}\label{eq:critdim-rescaled-labutin-local}
        -u_j(\xi)
        \le C_1\W_k^{\mu_j}(\xi,\frac{1}{2})
        +C_2\inf_{B(\xi,\frac{1}{4})}(-u_j),
        \qquad \xi\in B_1,
\end{equation}
and
\begin{equation}\label{eq:critdim-rescaled-infimum-bound}
\begin{aligned}
        \inf_{B(\xi,\frac{1}{4})}(-u_j)
        &=\inf_{B(x_j+a_j\xi,\frac{|x_j|}{16})}(-u) \\
        &
        \le C\log\frac1{|x_j+a_j\xi|}
        \le C\lambda_j,
        \qquad \xi\in B_1.
\end{aligned}
\end{equation}
Then \eqref{eq:critdim-rescaled-labutin-local} and
\eqref{eq:critdim-rescaled-infimum-bound} yield
\begin{equation}\label{eq:critdim-rescaled-pointwise-upper}
        -u_j(\xi)\le C_1\lambda_j+C_2\W_k^{\mu_j}(\xi,2),
        \qquad \xi\in B_1.
\end{equation}
Hence,
\begin{equation}\label{eq:critdim-large-wolff-contradiction-start}
        \frac{\W_k^{\mu_j}(0,2)}{\lambda_j}\longrightarrow\infty.
\end{equation}

We next estimate $\W_k^{\mu_j}(0,2)$.  The growth assumption
\eqref{eq:critical-source-growth}, the identity
$|\nabla_xu|=a_j^{-1}|\nabla_\xi u_j|$, and
\eqref{eq:critdim-rescaled-pointwise-upper} imply that, for $\xi\in B_1$,
\[
g_j(\xi)
\le C a_j^{n-\ell}|\nabla u_j(\xi)|^\ell
\exp\{C\lambda_j+C\W_k^{\mu_j}(\xi,2)\}.
\]
By Lemma~\ref{lem:hessian-measure-decomposition-origin},
$f(\cdot,u,\nabla u)\in L^1(B_{\frac12})$.  Since
$B(x_j,2a_j)\subset B_{\frac32|x_j|}$, a change of variables and
absolute continuity of the integral give
\[
        m_j:=\mu_j(B_2)
        =\int_{B(x_j,2a_j)}f(x,u(x),\nabla u(x))\,dx
        \longrightarrow0.
\]
Fix $A>0$.  For all sufficiently large $j$ with $m_j>0$,
$A m_j^{1/k}\le n/2$.  Applying
Proposition~\ref{prop:critdim-hessian-amt} with
$\Omega=B_2$, $D=4$, and $\delta=\frac12$, and using
$\W_k^{\mu_j}(\xi,2)\le\W_k^{\mu_j}(\xi,4)$, we obtain
\[
 \int_{B_2}e^{A \W_k^{\mu_j}(\xi,2)}\,d\xi\le C(A)
\qquad\text{for all sufficiently large }j,
\]
where $C(A)$ is independent of $j$.  If $m_j=0$, the potential
vanishes identically and the same bound is immediate.
Choose $q$ with $\ell<q<n$ when $\ell>0$, and choose any
$q\in(1,n)$ when $\ell=0$.  For all sufficiently large $j$,
$B(x_j,a_j)\subset B_{\frac{1}{2}}$.  The local gradient integrability gives
\[
\int_{B_1}|\nabla u_j|^q\,d\xi
=a_j^{q-n}\int_{B(x_j,a_j)}|\nabla u|^q\,dx
\le C a_j^{q-n}.
\]
Fix $s>1$ so that $s\ell<q$ if $\ell>0$; if $\ell=0$, take
any fixed $s>1$, and define
\[
 \Omega_j=\{\xi\in B_1:\W_k^{\mu_j}(\xi,2)\ge\lambda_j\}.
\]
On $\Omega_j$, we have $\lambda_j\le\W_k^{\mu_j}(\xi,2)$, so
the exponential factor in the bound for $g_j^s$ satisfies
\[
        \exp\{sC\lambda_j+sC\W_k^{\mu_j}(\xi,2)\}
        \le \exp\{2sC\W_k^{\mu_j}(\xi,2)\}.
\]
For $\ell>0$, H\"older's inequality and the preceding exponential
estimate, with the fixed coefficient $A=\frac{2sqC}{q-s\ell}$, give
\begin{align*}
        \int_{\Omega_j}g_j^s\,d\xi
        &\le C a_j^{s(n-\ell)}
        \int_{B_1}|\nabla u_j|^{s\ell}
        e^{2sC\W_k^{\mu_j}(\xi,2)}\,d\xi \\
        &\le C a_j^{s(n-\ell)}
        \left(\int_{B_1}|\nabla u_j|^q\,d\xi\right)^{\frac{s\ell}{q}}
        \left(\int_{B_2}
        e^{\frac{2sqC}{q-s\ell}\W_k^{\mu_j}(\xi,2)}\,d\xi
        \right)^{1-\frac{s\ell}{q}} \\
        &\le C a_j^{sn(1-\frac{\ell}{q})}\le C.
\end{align*}
When $\ell=0$, the same bound follows directly from the exponential
estimate with $A=2sC$, omitting the gradient factor.
Consequently, with $\gamma=n(1-\frac{1}{s})>0$,
\[
\int_{B(0,t)\cap\Omega_j}g_j\,d\xi\le Ct^\gamma,
\qquad 0<t\le2,
\]
and hence
\begin{equation}\label{eq:critdim-potential-split}
       \W_k^{\mu_j}(0,2)
        \le C+C\int_0^2
        \left(\int_{B(0,t)\setminus\Omega_j}g_j\,d\xi\right)^{\frac{1}{k}}
        \frac{dt}{t}.
\end{equation}
For $0<t\le1$, the definition of $\Omega_j$ gives
$\W_k^{\mu_j}(\xi,2)<\lambda_j$ on
$B(0,t)\setminus\Omega_j$.  Hence the preceding pointwise bound and
H\"older's inequality give
\[
\begin{aligned}
\int_{B(0,t)\setminus\Omega_j}g_j\,d\xi
&\le C a_j^{n-\ell}e^{C\lambda_j}
t^{n(1-\frac{\ell}{q})}
\left(\int_{B_1}|\nabla u_j|^q\,d\xi\right)^{\frac{\ell}{q}} \\
&\le Ce^{K\lambda_j}t^\beta,
\qquad 0<t\le1,
\end{aligned}
\]
where $\beta=n\left(1-\frac\ell q\right)>0$
and $K>0$ is independent of $j$.  For $1\le t\le2$, we use the
total mass instead:
\[
\int_1^2
\left(\int_{B(0,t)\setminus\Omega_j}g_j\,d\xi\right)^{\frac1k}
\frac{dt}{t}
\le \mu_j(B_2)^{\frac1k}\log2\le C.
\]
Here, for all sufficiently large $j$,
$\mu_j(B_2)=\mu(B(x_j,2a_j))\le\mu(B_{\frac12})<\infty$.
Fix $N>\frac{K}{\beta}$.  Since $e^{-N\lambda_j}<1$ for all
sufficiently large $j$, we may use the small-radius bound below this
radius and the total mass above it to obtain
\begin{equation}\label{eq:critdim-low-potential-integral}
\begin{aligned}
        \int_0^2
        \left(\int_{B(0,t)\setminus\Omega_j}g_j\,d\xi\right)^{\frac{1}{k}}
        \frac{dt}{t}
        &\le Ce^{\frac{K\lambda_j}{k}}
        \int_0^{e^{-N\lambda_j}}t^{\frac{\beta}{k}}\frac{dt}{t} \\
        &\quad+\left(\int_{B_2}g_j\,d\xi\right)^{\frac{1}{k}}
        \int_{e^{-N\lambda_j}}^2\frac{dt}{t} \\
        &\le Ce^{\frac{(K-N\beta)\lambda_j}{k}}
        +C(N\lambda_j+\log2) \\
        &\le C+C\lambda_j.
\end{aligned}
\end{equation}
Combining \eqref{eq:critdim-potential-split} and
\eqref{eq:critdim-low-potential-integral} yields
$\W_k^{\mu_j}(0,2)\le C+C\lambda_j$, contradicting
\eqref{eq:critdim-large-wolff-contradiction-start}.
\end{proof}

\subsection{Logarithmic blow-down and annular infimum asymptotic}

With the logarithmic upper bound established, we now identify the
coefficient of the logarithmic singularity.  We begin with a comparison
argument against the logarithmic fundamental solution.
\begin{lemma}
\label{lem:critdim-spherical-minimum}
For $0<r<1$, define
\[
        \gamma(r)=\min_{|x|=r}\frac{-u(x)}{\log\frac{1}{r}}.
\]
Then $\gamma(s)\le\gamma(r)$ whenever $0<s<r<1$.  Hence
\[
m=\lim_{r\to0}\gamma(r)
=\liminf_{x\to0}\frac{-u(x)}{\log\frac{1}{|x|}}
\]
exists, $m\ge0$, and
\begin{equation}\label{eq:critdim-global-comparison-lower}
        -u(x)\ge m\log\frac1{|x|},
        \qquad 0<|x|<1.
\end{equation}
\end{lemma}

\begin{proof}
For fixed $s\in(0,1)$, denote $h_s(x)=\gamma(s)\log|x|$.
By definition, we have $u\le h_s$ on $|x|=s$ while $u\le 0= h_s$ on $|x|=1$.  Moreover,
\[
       F_k[u]\ge0=F_k[h_s]
        \qquad\text{in }B_1\setminus\overline{B_s}.
\]
The comparison principle Theorem~3.1 in \cite{TrudingerWang1997} gives
$u\le h_s$ in the annulus, or equivalently
\begin{equation}\label{eq:critdim-comparison-from-inner-sphere}
        \frac{-u(x)}{\log\frac{1}{|x|}}\ge\gamma(s),
        \qquad s\le|x|<1.
\end{equation}
Taking the minimum on $|x|=r$ proves the asserted monotonicity.  The limit is
finite by Lemma~\ref{lem:critdim-exp-quotient-bound}; letting $s\to0$
in \eqref{eq:critdim-comparison-from-inner-sphere} gives
\eqref{eq:critdim-global-comparison-lower}.
\end{proof}

We next show that every  blow-down subsequence converges to the constant
$-m$.

\begin{lemma}
\label{lem:critdim-unique-blowdown}
Let $U_r(\xi)=\frac{u(r\xi)}{\log\frac{1}{r}}.$ Then $U_r\to-m$ in $L^1_{\mathrm{loc}}(\mathbb R^n\setminus\{0\})$ as $r\to0$.
\end{lemma}

\begin{proof}
Let $r_j\to0$ and $L_j=\log\frac{1}{r_j}$.  On every compact
annulus $K\Subset\mathbb R^n\setminus\{0\}$,
Lemma~\ref{lem:critdim-exp-quotient-bound} gives $0\le-U_{r_j}\le C_K$.
Compactness of $k$-convex functions and a diagonal argument therefore
give, after passage to a subsequence,
\[
        U_{r_j}\longrightarrow U
        \qquad\text{in }L^1_{\mathrm{loc}}(\mathbb R^n\setminus\{0\}),
\]
for some $U\in\Phi^k(\mathbb R^n\setminus\{0\})$, $U\le0$.  The local
lower bound for $U_{r_j}$ excludes the alternative that the compactness
limit is identically $-\infty$.  Since
$F_k[U_{r_j}]$ has density
$L_j^{-k}r_j^n f(r_j\xi,u(r_j\xi),\nabla u(r_j\xi))$, and its mass on $K$ is bounded by
\begin{equation}\label{eq:critdim-scaled-measure-zero}
        L_j^{-k}\int_{r_jK}f(x,u(x),\nabla u(x))\,dx
        \le L_j^{-k}\|f(\cdot,u,\nabla u)\|_{L^1(B_1)}\longrightarrow0,
\end{equation}
Weak continuity of Hessian measures gives $F_k[U]=0$ in
$\mathbb R^n\setminus\{0\}$.  Theorem~\ref{thm:k-hessian-punctured-bocher}
then implies that $U\equiv-c$ for some $c\ge0$.  It remains to
identify $c$.  

For brevity, write $U_j=U_{r_j}$, and choose
$\theta_j\in\mathbb S^{n-1}$ such that
$-U_j(\theta_j)=\gamma(r_j)$.
Let $\mathcal{A}=\{\frac{1}{2}<|\xi|<\frac{3}{2}\}$.
For all sufficiently large $j$ and every $\xi\in\mathcal A$, the lower bound
\eqref{eq:critdim-global-comparison-lower} gives
\[
\begin{aligned}
        -U_j(\xi)\ge m-\frac{m\log(3/2)}{L_j}
        =:b_j.
\end{aligned}
\]
 In particular,
$b_j\to m$. Set $  H_j=-U_j-b_j$.
Then $H_j\ge0$ and, since $U_j\in\Phi^k\subset\Phi^1$, the function
$H_j$ is superharmonic on $\mathcal A$.  
Since $B(\theta_j,\frac14)\Subset\mathcal A$, the mean-value inequality yields
\[
        0\le\frac1{|B(\theta_j,\frac14)|}\int_{B(\theta_j,\frac14)}H_j\,d\xi
        \le H_j(\theta_j)
        =\gamma(r_j)-b_j
        \longrightarrow0.
\]
Moreover, all the balls $B(\theta_j,\frac14)$ are contained in the fixed compact annulus
$ \mathcal A'=\{\frac23<|\xi|<\frac43\}
\Subset\mathcal A.$
Since $U_j\to-c$ in $L^1(\mathcal A')$,
\[
        \left|\frac1{|B(\theta_j,\frac14)|}\int_{B(\theta_j,\frac14)}(-U_j-c)\,d\xi\right|
        \le \frac1{|B(\theta_j,\frac14)|}
        \int_{\mathcal A'}|U_j+c|\,d\xi
        \longrightarrow0.
\]
Finally,
\[
        \frac1{|B(\theta_j,\frac14)|}\int_{B(\theta_j,\frac14)}H_j\,d\xi
        =\frac1{|B(\theta_j,\frac14)|}\int_{B(\theta_j,\frac14)}(-U_j-c)\,d\xi+c-b_j.
\]
Passing to the limit in this identity gives $0=0+c-m$, and hence
$c=m$.  Thus every subsequential limit of $U_r$ is the constant
$-m$.  Since the sequence $r_j$ was arbitrary, the proof is completed.
\end{proof}

\begin{corollary}
\label{cor:critdim-annular-infimum}
For every fixed $0<\theta<1$,
\begin{equation}\label{eq:critdim-annular-infimum}
        \lim_{x\to0}
        \frac{\displaystyle\inf_{B(x,\theta|x|)}(-u)}
        {\log\frac{1}{|x|}}=m.
\end{equation}
\end{corollary}

\begin{proof}
Write $x=r\omega$, $|\omega|=1$, and $L_r=\log(\frac{1}{r})$.  For
$y\in B(x,\theta r)$, the comparison lower bound gives
\[
        \frac{-u(y)}{L_r}
        \ge m\frac{\log(\frac{1}{|y|})}{L_r}
        \ge m\left(1-\frac{\log(1+\theta)}{L_r}\right).
\]
For the upper limit,
\[
        \frac{\inf_{B(x,\theta r)}(-u)}{L_r}
        =\inf_{B(\omega,\theta)}(-U_r)
        \le\frac1{|B(0,\theta)|}
        \int_{B(\omega,\theta)}(-U_r)\,d\xi.
\]
By
Lemma~\ref{lem:critdim-unique-blowdown}, the average on the right-hand side
converges to $m$ uniformly in $\omega$.
\end{proof}

\subsection{Local Wolff-potential decay and completion of the classification}

The preceding infimum asymptotics control the comparison term in the
local pointwise estimate.  It remains to prove that the local Wolff
potential of $\mu$ is negligible on the logarithmic scale.
\begin{lemma}
\label{lem:critdim-local-wolff-small}
\begin{equation}\label{eq:critdim-local-wolff-small}
        \lim_{x\to0}
        \frac{\W_k^\mu(x,\frac{|x|}{2})}{\log(\frac{1}{|x|})}=0.
\end{equation}
\end{lemma}

\begin{proof}
Write $h(y)=f(y,u(y),\nabla u(y))$, so that $d\mu=h\,dy$, and set
$r=|x|$.  We take $0<r<\frac14$; all constants below are independent
of $x$.  For $y\in B(x,\frac r2)$, we have
$\frac r2<|y|<\frac{3r}{2}<\frac12$.  Thus
Lemma~\ref{lem:critdim-exp-quotient-bound} directly gives
\begin{equation}\label{eq:critdim-local-pointwise-potential}
        -u(y)\le C_1\log\frac1r+C_2,
        \qquad y\in B(x,\frac r2),
\end{equation}
with $C_1>0$.  In particular,
\[
        e^{-\alpha u(y)}\le Cr^{-\kappa},
        \qquad \kappa=\alpha C_1>0.
\]

Choose $q_0$ with $\max\{1,\ell\}<q_0<n$.  The canonical extension
of $u$ belongs to $\Phi^k(B_1)$ by
Lemma~\ref{lem:hessian-measure-decomposition-origin}.  Since
$\frac{nk}{n-k}=n$, the local gradient estimate in Section~\ref{sec:prelim-removable} yields
\[
        \int_{B_{\frac12}}|\nabla u|^{q_0}\,dy<\infty.
\]
For $\ell>0$, the growth assumption \eqref{eq:critical-source-growth}
and H\"older's inequality therefore imply, for $0<s\le\frac r2$,
\[
\begin{aligned}
        \mu(B(x,s))
        &\le Cr^{-\kappa}\int_{B(x,s)}|\nabla u|^\ell\,dy \\
        &\le Cr^{-\kappa}
        \left(\int_{B_{\frac12}}|\nabla u|^{q_0}\,dy\right)^{\frac\ell{q_0}}
        s^{n(1-\frac\ell{q_0})}
        \le Cr^{-\kappa}s^\beta,
\end{aligned}
\]
where $\beta=n(1-\frac\ell{q_0})>0$.  When $\ell=0$, the same
bound follows directly by integrating $h\le Cr^{-\kappa}$, with
$\beta=n$.

Fix $N>1$ with $N\beta>\kappa$.  For all sufficiently small $r$,
$\rho=r^N<\frac r2$.  Since $k=\frac{n}{2}$, splitting the Wolff integral
at $\rho$ gives
\[
\begin{aligned}
        \W_k^\mu(x,\frac r2)
        &\le Cr^{-\frac{\kappa}{k}}
        \int_0^\rho s^{\frac\beta k}\frac{ds}{s}
        +\mu(B(x,\frac r2))^{\frac1k}
        \int_\rho^{\frac r2}\frac{ds}{s} \\
        &\le Cr^{\frac{N\beta-\kappa}{k}}
        +(N-1)\mu(B(x,\frac r2))^{\frac1k}\log\frac1r.
\end{aligned}
\]
Finally, $h\in L^1(B_{\frac12})$ by
Lemma~\ref{lem:hessian-measure-decomposition-origin}, and
$B(x,\frac r2)\subset B_{\frac{3r}{2}}$.  Consequently,
\[
        \mu(B(x,\frac r2))
        \le\int_{B_{\frac{3r}{2}}}h(y)\,dy\longrightarrow0
        \qquad\text{as }r\to0,
\]
uniformly in the direction of $x$.  Dividing the preceding potential
estimate by $\log\frac1r$ proves the claim.
\end{proof}

We now combine this decay with
Corollary~\ref{cor:critdim-annular-infimum} to obtain the exact logarithmic
limit and characterize the removable case.

\begin{proof}[Proof of Theorem~\ref{thm:critical-dimension-classification}]
Lemma~\ref{lem:critdim-spherical-minimum} provides $m\ge0$ and the global
lower bound.  Let $x\to0$ and $r=|x|$. Define
\[
        a_x=\inf_{B(x,\frac{3r}{4})}(-u),
        \qquad b_x=\inf_{B(x,\frac{r}{16})}(-u).
\]
 Apply
Lemma~\ref{lem:labutin-local-pointwise} to $\widetilde u:=u+a_x$.  Since
$\inf_{B(x,\frac{r}{16})}(-\widetilde u)=b_x-a_x$,
we obtain
\begin{equation}\label{eq:critdim-shifted-labutin}
        -u(x)-a_x
        \le C_1\W_k^\mu(x,\frac{r}{8})+C_2(b_x-a_x).
\end{equation}
By Corollary~\ref{cor:critdim-annular-infimum},
$\frac{a_x}{\log\frac{1}{r}}\to m$ and $\frac{b_x}{\log\frac{1}{r}}\to m$, while
Lemma~\ref{lem:critdim-local-wolff-small} gives
$\frac{\W_k^\mu(x,\frac{r}{8})}{\log\frac{1}{r}}\to0$.  Dividing
\eqref{eq:critdim-shifted-labutin} by $\log\frac{1}{r}$ yields
\[
        \limsup_{x\to0}\frac{-u(x)}{\log\frac{1}{r}}\le m.
\]
The reverse inequality follows from
\eqref{eq:critdim-global-comparison-lower}, proving
\eqref{eq:critical-dimension-limit}.

A removable singularity implies $m=0$.  Conversely, suppose $m=0$ and
write the canonical measure as in
\eqref{eq:critdim-measure-decomposition}.  If $M_0>0$,   for fixed $R>0$, by Lemma~\ref{lem:labutin-local-pointwise} we obtain
\[
\begin{aligned}
        -u(x)
        &\ge C(n,k)^{-1}
              \W_k^{F_k[u]}(x,\frac{R}{64})\\
        &\ge C(n,k)^{-1}M_0^{\frac{1}{k}}
              \int_{2|x|}^{\frac{R}{64}}\frac{dt}{t}
        =C(n,k)^{-1}M_0^{\frac{1}{k}}\log\frac{R}{128|x|}.
\end{aligned}
\]
This forces
$\liminf_{x\to0}-\frac{u(x)}{\log(\frac{1}{|x|})}>0$, contradicting $m=0$.
Hence $M_0=0$.

Choose $q_0$ with $\ell<q_0<n$ when $\ell>0$ and any
$q_0\in(1,n)$ when $\ell=0$.  Next choose $s>1$ sufficiently close to $1$
so that $\ell s<q_0$.  Since $-u(x)\le\varepsilon\log\frac1{|x|}$ for every
$\varepsilon>0$ and all sufficiently small $|x|$, we have
\[
        f(x,u(x),\nabla u(x))
        \le C|\nabla u(x)|^{\ell }|x|^{-\alpha\varepsilon }.
\]
For $\ell>0$, choose $\varepsilon>0$ so that
$\alpha\varepsilon s q_0/(q_0-\ell s)<n$.  H\"older's inequality and
$\nabla u\in L^{q_0}_{\loc}(B_1)$ then give
$f(\cdot,u,\nabla u)\in L^s$ near the origin.
For $\ell=0$, the same conclusion follows by direct integration provided
$\alpha\varepsilon s<n$.
We finish the proof by Proposition~\ref{thm:general-removable-lm}.
\end{proof}

\appendix
\section{Technical details for the complementary lower estimate}
\label{app:critical-lower-bound}

The proof of Proposition~\ref{prop:hessian-log-lower-bound} reduces the
lower estimate to
\eqref{eq:critical-bar-eventual-trapping}.  We establish it here.
The argument follows the upper-bound phase-plane analysis, so we concentrate
on the points affected by replacing the spherical minimum with the spherical
maximum.

Retain
\[
        \eta=\frac{n-2k}{k},
        \qquad
        q=\frac{nk}{n-2k},
        \qquad
        A_k=\frac{\binom{n-1}{k-1}}{\eta^{k-1}},
        \qquad
        t=\eta\log\frac{\eta}{r},
\]
let $\beta=q-k$, and define $\overline w$ by
\begin{equation}\label{eq:critical-bar-emden-fowler}
        e^t\overline w(t)=\overline u(r).
\end{equation}
The spherical maximum changes semiconcavity into semiconvexity, reverses the
jumps of the one-sided derivative, and reverses the corresponding inequalities.

\begin{lemma}\label{lem:critical-bar-w-semiconvexity}
The function $\overline w$ is locally semiconvex on
$(\eta\log\eta,\infty)$ and satisfies
\[
        \lim_{t\to\infty}\overline w(t)=0.
\]
If
\[
        \overline w_{tt}
        =\overline\nu_{ac}+\overline\nu_{pp}+\overline\nu_{sc},
\]
then $\overline\nu_{pp}$ and $\overline\nu_{sc}$ are nonnegative, and
$\overline\nu_{pp}$ is supported on the at most countable set
\[
        \overline{\mathcal J}
        :=\{t:\overline w_t^-(t)<\overline w_t^+(t)\}.
\]
\end{lemma}

\begin{proof}
Let $g_\theta$ be the functions introduced in the proof of
Lemma~\ref{lem:critical-w-semiconcavity}.  On each compact $t$-interval,
their second derivatives have a common lower bound, and
\[
        \overline w(t)=\sup_{\theta\in\mathbb S^{n-1}}g_\theta(t).
\]
Thus $\overline w$ is locally semiconvex.  Moreover,
Proposition~\ref{prop:no-dirac-pv} and
$e^{-t}=(r/\eta)^\eta$ give
\[
        \overline w(t)=\eta^{-\eta}r^\eta\overline u(r)\longrightarrow0.
\]
The remaining assertions follow from the semiconvex part of
Lemma~\ref{lem:semiconcavity-properties}.
\end{proof}

Set
\[
        \overline z(t)=\overline w_t^+(t),
        \qquad
        G(w,z)=-z-\frac{w^q}{A_k(w+z)^{k-1}}.
\]

\begin{lemma}\label{lem:critical-bar-phase-properties}
The trajectory $(\overline w,\overline z)$ has the following properties.
\begin{enumerate}[label=\textup{(\roman*)}]
\item
$\overline z$ is right-continuous, upper semicontinuous, and locally of
bounded variation.  Moreover, it is continuous on
$(\eta\log\eta,\infty)\setminus\overline{\mathcal J}$.

\item
If $\inf_{t\in[t_0,t_1]}(\overline w(t)+\overline z(t))>0$, then
\begin{align}
        \overline z(t_1)-\overline z(t_0)
        &\ge \overline w_t^-(t_1)-\overline w_t^+(t_0) \notag\\
        &\ge\int_{t_0}^{t_1}
        G(\overline w,\overline z)\,dt.
        \label{eq:critical-bar-phase-inequality}
\end{align}
\end{enumerate}
\end{lemma}

\begin{proof}
Part \textup{(i)} follows from
Lemmas~\ref{lem:semiconcavity-properties}
and~\ref{lem:critical-bar-w-semiconvexity}.  We prove \textup{(ii)}.

Set
\[
        \psi(\rho):=\min_{|x|=\rho}u(x)=-\overline u(\rho).
\]
Fix a point $t\in(t_0,t_1)$ at which $\overline w$ is twice
differentiable, and set
$
        r=\eta e^{-\frac{t}{\eta}}.
$
Then $\psi$ is twice differentiable at $r$.
Choose $x_r\in\partial B_r$ such that
$
        u(x_r)=\psi(r).
$
The radial function
$
        \Psi(x):=\psi(|x|)
$
touches $u$ from below at $x_r$.  Hence
$
        D^2\Psi(x_r)\le D^2u(x_r).
$

At almost every $t\in(t_0,t_1)$, one has
$\overline w_t=\overline z$ and
$\psi'(r)=\frac{\eta e^t}{r}
\bigl(\overline w+\overline w_t\bigr)$.
Since $\overline w+\overline z>0$, the tangential eigenvalue
$b=\psi'(r)/r$ of $D^2\Psi(x_r)$ is positive.  Writing the radial
eigenvalues as $(a,b,\ldots,b)$, we have
\[
        \sigma_j(a,b,\ldots,b)
        =\binom{n-1}{j-1}b^{j-1}
        \left(a+\frac{n-j}{j}b\right),
        \qquad 1\le j\le k.
\]
If $\sigma_k(a,b,\ldots,b)\le0$, then
\[
        F_k[\Psi](x_r)\le0<\overline u(r)^q.
\]
If $\sigma_k(a,b,\ldots,b)>0$, then
$\frac{n-j}{j}\ge\frac{n-k}{k}$ for $1\le j\le k$ shows that
$D^2\Psi(x_r)\in\Gamma_k$.  Since $u$ is $k$-admissible,
the monotonicity of $F_k$ on $\Gamma_k$ therefore gives
\[
        F_k[\Psi](x_r)
        \le F_k[u](x_r)
        =(-u(x_r))^q
        =\overline u(r)^q.
\]
Thus the same inequality holds in either case.

Since $\overline u(r)=e^t\overline w(t)$ and
\[
        F_k[\Psi]
        =-A_ke^{qt}
        \bigl(\overline w+\overline w_t\bigr)^{k-1}
        \bigl(\overline w_t+\overline w_{tt}\bigr),
\]
 it follows that, at almost
every $t\in(t_0,t_1)$,
\[
        \overline f_{ac}(t)
        \ge
        -\overline z(t)
        -\frac{\overline w(t)^q}
        {A_k\bigl(\overline w(t)+\overline z(t)\bigr)^{k-1}}
        =G\bigl(\overline w(t),\overline z(t)\bigr),
\]
where
\[
        \overline\nu_{ac}=\overline f_{ac}(t)\,dt
\]
is the absolutely continuous part of $\overline w_{tt}$.

By Lemma~\ref{lem:critical-bar-w-semiconvexity}, the singular parts
$\overline\nu_{pp}$ and $\overline\nu_{sc}$ are nonnegative.
The assumed positive lower bound gives
$G(\overline w,\overline z)\in L^1(t_0,t_1)$; consequently,
\[
        \overline w_{tt}
        \ge G(\overline w,\overline z)\,dt
\]
as an inequality of Radon measures on $(t_0,t_1)$.  The generalized
Newton--Leibniz formula therefore yields
\[
        \overline w_t^-(t_1)-\overline w_t^+(t_0)
        =\overline w_{tt}\bigl((t_0,t_1)\bigr)
        \ge\int_{t_0}^{t_1}G(\overline w,\overline z)\,dt.
\]
Finally,
\[
        \overline z(t_1)=\overline w_t^+(t_1)
        \ge\overline w_t^-(t_1),
        \qquad
        \overline z(t_0)=\overline w_t^+(t_0),
\]
which gives the first inequality in
\eqref{eq:critical-bar-phase-inequality}.
\end{proof}

\begin{lemma}\label{lem:critical-bar-linear-wedge}
The function $\overline w$ does not satisfy
$\overline w(t)=O(e^{-\kappa t})$ for any $\kappa>0$.  Moreover, for
every $0<\theta<\frac12$, there exists $T_\theta>0$ such that
\begin{equation}\label{eq:critical-bar-wedge}
        \overline z(t)>-\theta\overline w(t)
        \qquad(t\ge T_\theta).
\end{equation}
\end{lemma}

\begin{proof}
Suppose that $\overline w(t)\le Ce^{-\kappa t}$ for all large $t$, and let
$\kappa_0=\min\{\kappa,\frac12\}$.  Since
$e^t\overline w(t)=\overline u(r)$,
\[
        -u(x)\le\overline u(|x|)
        \le C|x|^{-\eta(1-\kappa_0)}.
\]
Choose $a$ with $q<a<q/(1-\kappa_0)$.  Then
$-u\in L^a(B_1)$, and Corollary~\ref{cor:continuous-extension} makes the
singularity removable, a contradiction.

Fix $0<\theta<\frac12$.  There are arbitrarily large times at which
$\overline z> -\theta\overline w$; otherwise
$\overline w_t=\overline z\le-\theta\overline w$ almost everywhere for
all large $t$, which would imply exponential decay of $\overline w$.
Since $\overline w(t)\to0$, we may therefore choose $t_0$ so large that $ \overline z(t_0)>-\theta\overline w(t_0)$
and
\[
        \frac{\overline w(t)^\beta}
        {A_k\theta(1-\theta)^{k-1}}\le\frac12
        \qquad\text{for every }t\ge t_0.
\]
Suppose that the trajectory later exits the region
$z>h(w)$, where $h(w)=-\theta w$, and let $t_1>t_0$ be the first
contact time.  Then
\[
        \overline z(t_1)=h(\overline w(t_1)),
        \qquad
        \overline z(t)>h(\overline w(t))
        \quad(t_0\le t<t_1).
\]
Moreover, by Lemma~\ref{lem:critical-bar-w-semiconvexity},
at $t=t_1$ we have
\[
        \overline z(t_1)=-\theta\overline w(t_1)<0,
        \qquad
        \overline w(t_1)+\overline z(t_1)
        =(1-\theta)\overline w(t_1)>0.
\]
By continuity at $t_1$, for some $\delta>0$ we have
$\overline z<0$ and
$\inf_{[t_1-\delta,t_1]}(\overline w+\overline z)>0$.
Thus $\overline w$ is strictly decreasing there.
For $t<t_1$ sufficiently close to $t_1$,
\[
        \frac{\overline z(t)-\overline z(t_1)}
        {\overline w(t)-\overline w(t_1)}
        \ge
        \frac{h(\overline w(t))-h(\overline w(t_1))}
        {\overline w(t)-\overline w(t_1)}.
\]
Consequently,
\[
        \liminf_{t\to t_1^-}
        \frac{\overline z(t)-\overline z(t_1)}
        {\overline w(t)-\overline w(t_1)}
        \ge h'(\overline w(t_1))=-\theta.
\]
On the other hand, applying
Lemma~\ref{lem:critical-bar-phase-properties}\textup{(ii)} on $[t,t_1]$,
dividing by
\[
        \overline w(t)-\overline w(t_1)
        =\int_t^{t_1}(-\overline z(s))\,ds>0,
\]
and letting $t\to t_1^-$, we obtain
\begin{align*}
        \limsup_{t\to t_1^-}
        \frac{\overline z(t)-\overline z(t_1)}
        {\overline w(t)-\overline w(t_1)}
        &\le
        \frac{G(\overline w(t_1),\overline z(t_1))}
        {\overline z(t_1)}\\
        &=-1+
        \frac{\overline w(t_1)^\beta}
        {A_k\theta(1-\theta)^{k-1}}
        \le-\frac12.
\end{align*}
This contradicts $-\theta>-\frac12$.  Taking $T_\theta=t_0$
proves \eqref{eq:critical-bar-wedge}.
\end{proof}

Fix $0<\varepsilon<1$, 
\[
        B_\varepsilon=\frac1{(1-\varepsilon)A_k},
\]
and choose $\delta_\varepsilon\in(0,1)$ so small that
\begin{equation}\label{eq:critical-bar-delta-choice}
\begin{gathered}
 B_\varepsilon\delta_\varepsilon^\beta<\frac12,
 \qquad
 B_\varepsilon(\beta+1)\delta_\varepsilon^\beta<\frac\varepsilon4,\\
 -1+\frac{1-\varepsilon}
 {(1-B_\varepsilon\delta_\varepsilon^\beta)^{k-1}}
 <-\frac\varepsilon2.
\end{gathered}
\end{equation}
Define
\begin{equation}\label{eq:critical-bar-region}
        \mathcal S_\varepsilon
        =\left\{(w,z):
        0<w\le\delta_\varepsilon,
        \quad z>-B_\varepsilon w^{\beta+1}
        \right\}.
\end{equation}

\begin{lemma}\label{lem:critical-bar-invariant-region}
There exists $T_\varepsilon>0$ such that, if
$
        (\overline w(t_0),\overline z(t_0))\in\mathcal S_\varepsilon
$
for some $t_0\ge T_\varepsilon$, then
\[
        (\overline w(t),\overline z(t))\in\mathcal S_\varepsilon
        \qquad(t\ge t_0).
\]
\end{lemma}

\begin{proof}
Set $h(w)=-B_\varepsilon w^{\beta+1}$.
Choose $T_\varepsilon$ so that
$0<\overline w(t)<\delta_\varepsilon$ for $t\ge T_\varepsilon$.
Suppose for contradiction that the trajectory enters
$\mathcal S_\varepsilon$ at $t_0\ge T_\varepsilon$ and later exits.
There is then a first contact time $t_1>t_0$ such that
\[
        \overline z(t_1)=h(\overline w(t_1)),
        \qquad
        \overline z(t)>h(\overline w(t))
        \quad(t_0\le t<t_1).
\]
Moreover, at $t=t_1$ we have by Lemma~\ref{lem:critical-bar-w-semiconvexity}, 
\[
        \overline z(t_1)=-B_\varepsilon\overline w(t_1)^{\beta+1}<0
\]
and by \eqref{eq:critical-bar-delta-choice},
\[
        \overline w(t_1)+\overline z(t_1)
        =\overline w(t_1)
        \bigl(1-B_\varepsilon\overline w(t_1)^\beta\bigr)>0.
\]
By continuity at $t_1$, for some $\delta>0$ we have
$\overline z<0$ and
$\inf_{[t_1-\delta,t_1]}(\overline w+\overline z)>0$.
Thus $\overline w$ is strictly decreasing there.
For $t<t_1$ sufficiently close to $t_1$,
\[
        \frac{\overline z(t)-\overline z(t_1)}
        {\overline w(t)-\overline w(t_1)}
        \ge
        \frac{h(\overline w(t))-h(\overline w(t_1))}
        {\overline w(t)-\overline w(t_1)}.
\]
Hence
\begin{equation}\label{eq:critical-bar-first-contact-lower}
        \liminf_{t\to t_1^-}
        \frac{\overline z(t)-\overline z(t_1)}
        {\overline w(t)-\overline w(t_1)}
        \ge h'(\overline w(t_1))
        >-\frac\varepsilon4.
\end{equation}
On the other hand, applying
Lemma~\ref{lem:critical-bar-phase-properties}\textup{(ii)} on $[t,t_1]$,
dividing by
\[
        \overline w(t)-\overline w(t_1)
        =\int_t^{t_1}(-\overline z(s))\,ds>0,
\]
and letting $t\to t_1^-$, we obtain
\begin{align*}
        \limsup_{t\to t_1^-}
        \frac{\overline z(t)-\overline z(t_1)}
        {\overline w(t)-\overline w(t_1)}
        &\le
        \frac{G(\overline w(t_1),\overline z(t_1))}
        {\overline z(t_1)}\\
        &=-1+
        \frac{1-\varepsilon}
        {\bigl(1-B_\varepsilon\overline w(t_1)^\beta\bigr)^{k-1}}
        <-\frac\varepsilon2,
\end{align*}
by \eqref{eq:critical-bar-delta-choice}.  This contradicts
\eqref{eq:critical-bar-first-contact-lower}.
\end{proof}

Choose $0<\theta<\frac12$ so small that
\begin{equation}\label{eq:critical-bar-theta-choice}
        \frac{1-\varepsilon}{(1-\theta)^{k-1}}
        \le1-\frac\varepsilon2.
\end{equation}

\begin{lemma}\label{lem:critical-bar-eventual-entry}
There exists $t_0\ge T_\varepsilon$ such that
\[
        (\overline w(t_0),\overline z(t_0))\in\mathcal S_\varepsilon.
\]
\end{lemma}

\begin{proof}
Increase $T_\varepsilon$, if necessary, so that
\eqref{eq:critical-bar-wedge} holds and
$0<\overline w<\delta_\varepsilon$ on
$[T_\varepsilon,\infty)$.  If the conclusion failed, then
\[
        -\theta\overline w<\overline z
        \le-B_\varepsilon\overline w^{\beta+1}<0
        \qquad(t\ge T_\varepsilon).
\]
Set
\[
        Y=-\frac{\overline z}{\overline w^{\beta+1}}
        \ge B_\varepsilon,
        \qquad
        \Theta=-\frac{\overline z}{\overline w}<\theta.
\]
Then $\overline w+\overline z>(1-\theta)\overline w>0$, with a positive
lower bound on every compact time interval, and
\eqref{eq:critical-bar-theta-choice} gives
\[
        \frac{G(\overline w,\overline z)}{\overline z}
        =-1+\frac1{A_kY(1-\Theta)^{k-1}}
        \le-\frac\varepsilon2.
\]
Since $\overline z<0$, the measure inequality in
Lemma~\ref{lem:critical-bar-phase-properties}\textup{(ii)} implies, with
$s=-\overline z>0$,
\[
        ds+\frac\varepsilon2s\,dt\le0.
\]
Thus $s(t)\le Ce^{-\varepsilon t/2}$.  Because
$\overline w_t=-s$ almost everywhere and $\overline w(t)\to0$,
\[
        \overline w(t)=\int_t^\infty s(\tau)\,d\tau
        \le Ce^{-\varepsilon t/2},
\]
contradicting Lemma~\ref{lem:critical-bar-linear-wedge}.
\end{proof}

Lemmas~\ref{lem:critical-bar-invariant-region} and
\ref{lem:critical-bar-eventual-entry} prove
\eqref{eq:critical-bar-eventual-trapping}, completing the remaining part of the proof of Proposition~\ref{prop:hessian-log-lower-bound}.

\bigskip
\noindent \textbf{Acknowledgements:} The first named author is supported by Natural Science Foundation of Tianjin (Grant No. 22JCJQJC00130), the Fundamental Research Funds for the Central Universities.

\bigskip

\noindent Shiguang Ma: School of Mathematical Sciences and LPMC, Nankai University\\ 
Tianjin 300071, China \\
\textit{e-mail}: msgdyx8741@nankai.edu.cn \\

\medskip

\noindent Shengyang Zang: School of Mathematical Sciences, Nankai University\\ 
Tianjin 300071, China\\
\textit{e-mail}: 1120230032@mail.nankai.edu.cn

\end{document}